\documentclass{mathincs}
\usepackage{graphicx}
\usepackage[T1]{fontenc}
\usepackage{amsfonts}
\usepackage{amsmath,amssymb,mathtools,bbm,url}
\usepackage{todonotes}
\usepackage[normalem]{ulem}

\usepackage[shortlabels]{enumitem}
\usepackage{url}
\usepackage{algorithm}
\usepackage{algorithmic}

\newtheorem{theorem}{Theorem}
\newtheorem{proposition}[theorem]{Proposition}
\theoremstyle{definition}
\newtheorem{definition}[theorem]{Definition}
\theoremstyle{remark}
\newtheorem{remark}[theorem]{Remark}
\newtheorem{example}[theorem]{Example}

\newcommand{\NN}{\mathbbm{N}}

\newcommand{\RR}{\mathbbm{R}}
\newcommand{\ZZ}{\mathbbm{Z}}

\newcommand{\QQQ}{\mathcal{Q}}
\newcommand{\RRR}{\mathcal{R}}
\newcommand{\SSS}{\mathcal{S}}

\newcommand{\Eta}{\mathrm{H}}
\newcommand{\perm}[1]{\pi(#1)}

\newcommand{\vv}{\mathbf{v}}
\newcommand{\xx}{\mathbf{x}}
\newcommand{\yy}{\mathbf{y}}

\newcommand{\false}{\operatorname{false}}
\newcommand{\true}{\operatorname{true}}

\newcommand{\dotneq}{\mathrel{\dot\neq}}
\newcommand{\dotequal}{\mathrel{\dot=}}

\begin{document}

\title[Real Singularities of Implicit ODEs]{A Logic Based Approach to
  Finding Real Singularities of Implicit Ordinary Differential Equations}
\author{Werner M. Seiler}
\address{Institut f\"ur Mathematik, Universit\"at Kassel, 34109 Kassel,
  Germany}
\email{seiler@mathematik.uni-kassel.de}
\author{Matthias Sei\ss}
\address{Institut f\"ur Mathematik, Universit\"at Kassel, 34109 Kassel,
  Germany}
\email{seiss@mathematik.uni-kassel.de}
\author{Thomas Sturm}
\address{CNRS, Inria, and the University of Lorraine, Nancy, France\\
  MPI Informatics and Saarland University, Saarbr\"ucken, Germany}
\email{thomas.sturm@loria.fr, sturm@mpi-inf.mpg.de}

\begin{abstract}
  We discuss the effective computation of geometric singularities of
  implicit ordinary differential equations over the real numbers using
  methods from logic.  Via the Vessiot theory of differential equations,
  geometric singularities can be characterised as points where the
  behaviour of a certain linear system of equations changes.  These points
  can be discovered using a specifically adapted parametric generalisation
  of Gaussian elimination combined with heuristic simplification techniques
  and real quantifier elimination methods.  We demonstrate the relevance
  and applicability of our approach with computational experiments using a
  prototypical implementation in \textsc{Reduce}.
\end{abstract}

\maketitle

\section{Introduction}

Implicit differential equations, i.\,e.\ equations which are not solved for
a derivative of highest order, appear in many applications. In particular,
the so-called differential algebraic equations (DAE) may be considered as a
special case of implicit equations.\footnote{Differential algebraic
  equations owe their name to the fact that in a solved form they often
  comprise both differential equations and ``algebraic'' equations (meaning
  equations in which no derivatives appear). This should not be confused
  with (semi)algebraic differential equations, the main topic of this work,
  where the ``algebraic'' refers to the fact that only polynomial
  nonlinearities are permitted (see below).} Compared with equations in
solved form, implicit equations are more complicated to analyse and show a
much wider range of phenomena. Already basic questions about the existence
and uniqueness of solutions of an initial value problem become much more
involved.  One reason is the possible appearance of singularities. Note
that we study in this work singularities of the differential equations
themselves (defined below in a geometric sense) and not singularities of
individual solutions like poles.

Our approach to singularities of differential equations is conceptually
based on the theory of singularities of maps between smooth manifolds (as
e.\,g.\ described in \cite{agv:sing1,gg:stable}), i.\,e.\ of a differential
topological nature.  Within this theory, the main emphasis has
traditionally been on classifying possible types of singularities and on
corresponding normal forms, see e.\,g.\ \cite{ld:singgen,aad:normform}.
Nice introductions can be found in \cite{via:geoode} or \cite{aor:bising}.
By contrast, we are here concerned with the \emph{effective} detection of
all geometric singularities of a given implicit ordinary differential
equation.  This requires the additional use of techniques from differential
algebra \cite{ko:daag,ritt:da} and algebraic geometry \cite{clo:iva}.

In \cite{lrss:gsade}, the first two authors developed together with
collaborators a novel framework for the analysis of \emph{algebraic}
differential equations, i.\,e.\ differential equations (and inequations)
described by differential polynomials, which combines ideas and techniques
from differential algebra, differential geometry and algebraic
geometry.\footnote{For scalar ordinary differential equations of first
  order, a somewhat similar theory was developed by Hubert \cite{eh:degen}.
  The approach in \cite{lrss:gsade} covers much more general situations
  including systems of arbitrary order and partial differential equations.}
A key role in this new effective approach is played by the \emph{Thomas
  decomposition} which exists in an algebraic version for algebraic systems
and in a differential version for differential systems.  Both were first
introduced by Thomas \cite{th:ds,th:sr} and later rediscovered by Gerdt
\cite{vpg:decomp}; an implementation in \textsc{Maple} is described in
\cite{bglr:thomas} (see also \cite{glhr:tdds}).  Unfortunately, the
algorithms behind the Thomas decomposition require that the underlying
field is algebraically closed.  Hence, it is always assumed in
\cite{lrss:gsade} that a \emph{complex} differential equation is treated.
However, most differential equations appearing in applications are real.
The main goal of this work is to adapt the framework of \cite{lrss:gsade}
to \emph{real} ordinary differential equations.

The approach in \cite{lrss:gsade} consists of a differential and an
algebraic step.  For the prepatory differential step, one may continue to
use basic differential algebraic algorithms (for example the differential
Thomas decomposition).  A key task of the differential step is to exhibit
all integrability conditions which may be hidden in the given system and
for this the base field does not matter.  In this work, we are mainly
concerned with presenting an alternative for the algebraic step -- where
the actual identification of the singularities happens -- which is valid
over the real numbers.

Our use of real algebraic geometry combined with computational logic has
several benefits.  In a complex setting, one may only consider inequations.
Over the reals, also the treatment of inequalities like positivity
conditions is possible which is important for many applications e.\,g.\ in
biology and chemistry.  We will extend the approach from \cite{lrss:gsade}
by generalising the notion of an algebraic differential equation used in
\cite{lrss:gsade} to \emph{semi}algebraic differential equations, which
allow for arbitrary inequalities.  As a further improvement, we will make
stronger use of the fact that the detection of singularities represents
essentially a linear problem.  This will allow us to avoid some redundant
case distinctions that are unavoidable in the approach of
\cite{lrss:gsade}, as they must appear in any algebraic Thomas
decomposition, although they are irrelevant for the detection of
singularities.

The article is structured as follows.  The next section firstly exhibits
some basics of the geometric theory of (ordinary) differential equations.
We then recapitulate the key ideas behind the differential step of
\cite{lrss:gsade} and encapsulate the key features of the outcome in the
improvised notion of a ``well-prepared'' system.  Finally, we define the
geometric singularities that are studied here.  In the third section, we
develop a Gauss algorithm for linear systems depending on parameters with
certain extra features and rigorously prove its correctness.  The fourth
section represents the core of our work.  We show how finding geometric
singularities can essentially be reduced to the analysis of a parametric
linear system and present then an algorithm for the automatic detection of
all real geometric singularities based on our parametric Gauss algorithm.
The fifth section demonstrates the relevance of our algorithm by applying
it to some basic examples some of which stem from the above mentioned
classifications of all possible singularities of scalar first-order
equations.  Although these examples are fairly small, it becomes evident
how our logic based approach avoids some unnecessary case distinctions made
by the algebraic Thomas decomposition.

\section{Geometric Singularities of Implicit Ordinary Differential
  Equations}

We use the basic set-up of the geometric theory of differential equations
following \cite{wms:invol} to which we refer for more details. For a system
of ordinary differential equations of order $\ell$ in $m$ unknown
real-valued functions $u_{\alpha}(t)$ of the independent real variable $t$,
we construct over the trivial fibration
$\pi=\mathrm{pr}_{1}:\RR\times\RR^{m}\rightarrow\RR$ the \emph{$\ell$th
  order jet bundle $J_{\ell}\pi$}. For our purposes, it is sufficient to
imagine $J_{\ell}\pi$ as an affine space diffeomorphic to
$\RR^{(\ell+1)m+1}$ with coordinates
$(t,\mathbf{u},\mathbf{\dot{u}}\dots,\mathbf{u}^{(\ell)})$ corresponding to
the independent variable $t$, the $m$ dependent variables
$\mathbf{u}=(u_{1},\dots,u_{m})$ and the derivatives of the latter ones up
to order $\ell$. We denote by $\pi^{\ell}:J_{\ell}\pi\rightarrow\RR$ the
canonical projection on the first coordinate.  The contact structure is a
geometric way to encode the different roles played by the different
variables, i.\,e.\ that $t$ is the independent variable and that
$u_{\alpha}^{(i)}$ denotes the derivative of $u_{\alpha}^{(i-1)}$ with
respect to $t$.  We describe the contact structure by the \emph{contact
  distribution} $\mathcal{C}^{(\ell)}\subset TJ_{\ell}\pi$ which is spanned
by one $\pi^{\ell}$-transversal and $m$ $\pi^{\ell}$-vertical vector
fields:\footnote{A vector field $X$ is $\pi^{\ell}$-vertical, if at every
  point $\rho\in J_{\ell}\pi$ we have
  $X_{\rho}\in\ker{T_{\rho}\pi^{\ell}}$; otherwise it is
  $\pi^{\ell}$-transversal.}
\begin{displaymath}
  C^{(\ell)}_{\mathrm{trans}}=\partial_{t} +
      \sum_{i=1}^{\ell}\sum_{\alpha=1}^{m}
      u_{\alpha}^{(i)}\cdot\partial_{u_{\alpha}^{(i-1)}}\,,\qquad
  C^{(\ell)}_{\alpha}=\partial_{u_{\alpha}^{(\ell)}}\quad (\alpha=1,\dots,m)\,.
\end{displaymath}
The transversal field essentially corresponds to a geometric version of the
chain rule and the vertical fields are needed because we must cut off the
chain rule at a finite order, since in $J_{\ell}\pi$ no variables exist
corresponding to derivatives of order $\ell+1$ required for the next terms
in the chain rule.

We can now rigorously define the class of differential equations that will
be studied in this work.  Note that in the geometric theory one does not
distinguish between a scalar equation and a system of equations, as a
differential equation is considered as a single geometric object
independent of its codimension.  In \cite{lrss:gsade}, an algebraic jet set
of order $\ell$ is defined as a locally Zariski closed subset of
$J_{\ell}\pi$, i.\,e.\ as the set theoretic difference of two varieties.
This approach reflects the fact that over the complex numbers only
equations and inequations are allowed.  Over the real numbers, one would
like to include arbitrary inequalities like for example positivity
conditions.  Thus it is natural to proceed from algebraic to semialgebraic
geometry.  Recall that a semialgebraic subset of $\RR^{n}$ is the solution
set of a Boolean combination of conditions of the form $f=0$ or
$f\diamond0$ where $f$ is a polynomial in $n$ variables and $\diamond$
stands for some relation in $\{<,>,\leq,\geq,\neq\}$ (see e.\,g.\
\cite[Chap.~2]{bcr:realag}).

\begin{definition}\label{def:semi}
  A \emph{semialgebraic jet set} of order $\ell$ is a semialgebraic subset
  $\mathcal{J}_{\ell}\subseteq J_{\ell}\pi$ of the $\ell$th order jet
  bundle.  Such a set is a \emph{semialgebraic differential equation}, if
  in addition the Euclidean closure of $\pi^{\ell}(\mathcal{J}_{\ell})$ is
  the whole base space $\RR$.
\end{definition}

In the traditional geometric theory, a differential equation is a fibred
submanifold of $J_{\ell}\pi$ such that the restriction of $\pi^{\ell}$ to
it defines a surjective submersion. The latter condition excludes any kind
of singularities and is thus dropped in our approach.  We replace the
submanifold by a semialgebraic and thus in particular constructible set,
i.\,e.\ a finite union of locally Zariski closed sets.  This is on the one
hand more restrictive, as only polynomial equations and inequalities are
allowed.  On the other hand, it is more general, as a semialgebraic set may
have singularities in the sense of algebraic geometry.  We will call such
points \emph{algebraic singularities} of the semialgebraic differential
equation $\mathcal{J}_{\ell}$ to distinguish them from the geometric
singularities on which we focus in this work.

The additional closure condition imposed in Definition \ref{def:semi} for a
semialgebraic differential equation ensures that the semialgebraic
differential system defining it does not contain equations depending solely
on $t$ and thus that $t$ represents indeed an independent variable.
Nevertheless, we admit that certain values of $t$ are not contained in the
image $\pi^{\ell}(\mathcal{J}_{\ell})$.  This relaxation compared with the
standard geometric theory allows us to handle equations like $t\dot{u}=1$
where the point $t=0$ is not contained in the projection.  We use the
Euclidean closure instead of the Zariski one, as for a closer analysis of
the solution behaviour around such a point (which we will not do in this
work) it is of interest to consider the point as the limit of a sequence of
points in $\pi^{\ell}(\mathcal{J}_{\ell})$.

A (sufficiently often differentiable) function
$\mathbf{g}:\mathcal{I}\subseteq\RR\rightarrow\RR^{m}$ defined on some
interval $\mathcal{I}$ is a \emph{(local) solution} of the semialgebraic
differential equation $\mathcal{J}_{\ell}\subset J_{\ell}\pi$, if its
prolonged graph, i.\,e.\ the image of the curve
$\gamma_{\mathbf{g}}:\mathcal{I}\rightarrow J_{\ell}\pi$ given by
$t\mapsto \bigl(t,\mathbf{g}(t), \mathbf{\dot{g}}(t), \dots,
\mathbf{g}^{(\ell)}(t)\bigr)$ lies completely in the set
$\mathcal{J}_{\ell}$.  This definition of a solution represents simply a
geometric version of the usual one.  Figure~\ref{fig:ode} shows the
semialgebraic differential equation $\mathcal{J}_{1}$ which is defined by
the scalar first-order equation $\dot{u}-tu^{2}=0$ together with some of
its prolonged solutions.  $\mathcal{J}_{1}$ is a classical example of a
differential equation with so-called movable singularities: its solutions
are given by $u(t)=2/(c-t^{2})$ with an arbitrary constant $c\in\RR$ and
each solution with a positive $c$ becomes singular after a finite time.
However, this differential equation does \emph{not} exhibit the kind of
singularities that we will be studying in this work.  We are concerned with
singularities of the differential equation itself and not with
singularities of individual solutions.

\begin{figure}
  \centering
  \includegraphics[width=0.5\textwidth]{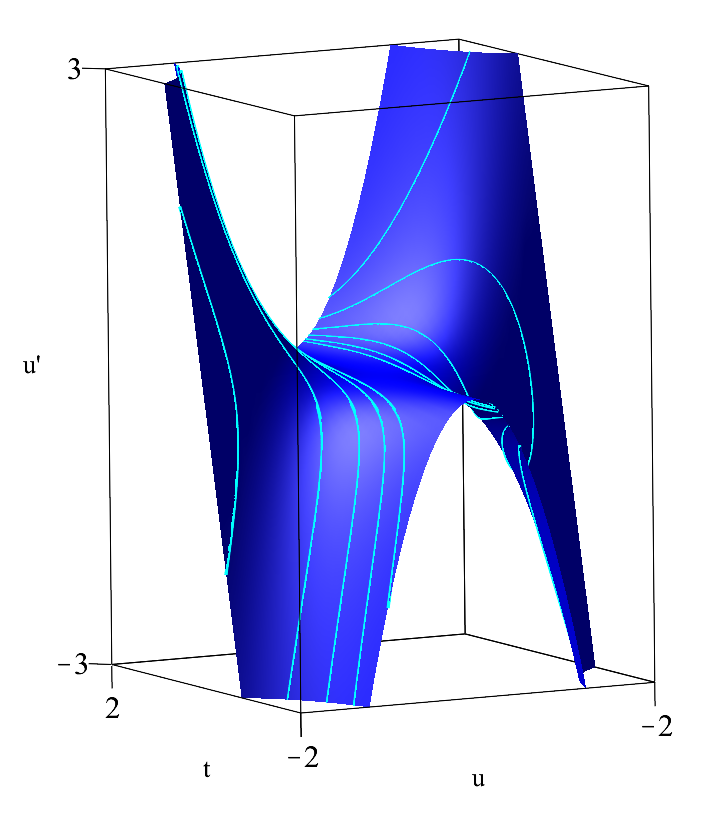}
  \caption{A semialgebraic differential equation with some prolonged solutions}
  \label{fig:ode}
\end{figure}
  
We call a semialgebraic jet set $\mathcal{J}_{\ell}\subseteq J_{\ell}\pi$
\emph{basic}, if it can be described by a finite set of equations $p_{i}=0$
and a finite set of inequalities $q_{j}>0$ where $p_{i}$ and $q_{j}$ are
polynomials in the coordinates
$(t,\mathbf{u},\mathbf{\dot{u}}\dots,\mathbf{u}^{(\ell)})$.  We call such a
pair of sets a \emph{basic semialgebraic system} on $J_{\ell}\pi$.  It
follows from an elementary result in real algebraic geometry
\cite[Prop.~2.1.8]{bcr:realag} that any semialgebraic jet set can be
expressed as a union of finitely many basic semialgebraic jet sets.  We
will always assume that our sets are given in this form and study each
basic semialgebraic system separately, as for some steps in our analysis it
is crucial that at least the equation part of the system is a pure
conjunction.

To obtain correct and meaningful results with our approach, we need some
further assumptions on the basic semialgebraic differential systems we are
studying.  More precisely, the systems have to be carefully prepared using
a procedure essentially corresponding to the differential step of the
approach developed in \cite{lrss:gsade} and the subsequent transformation
from a differential algebraic formulation to a geometric one. Otherwise,
hidden integrability conditions or other subtle problems may lead to false
results.  We present here only a very brief description of this procedure
and refer for all details and an extensive discussion of the underlying
problems to \cite{lrss:gsade}.  We use in the sequel some basic notions
from differential algebra \cite{ko:daag,ritt:da} and the Janet--Riquier
theory of differential equations \cite{ja:lec,riq:edp} which can be found
in modern form for example in \cite{dr:habil} to which we refer for
definitions of all unexplained terminology and for background information.

The starting point of our analysis will always be a basic semialgebraic
system with equations $p_{i}=0$ $(1\leq i\leq r)$ and inequalities
$q_{j}>0$ $(1\leq j\leq s)$.  We call such a system \emph{differentially
  simple} with respect to some orderly ranking $\prec$, if it satisfies the
following three conditions:
\begin{enumerate}[(i)]
\item all polynomials $p_{i}$ and $q_{j}$ are non-constant and have
  pairwise different leaders,
\item no leader of an inequality $q_{j}$ is a derivative of the leader of
  an equation $p_{i}$,
\item away from the variety defined by the vanishing of all the initials
  and all the separants of the polynomials $p_{i}$, the equations
  define a passive differential system for the Janet division.
\end{enumerate}
The last condition ensures the absence of hidden integrability conditions
and thus the existence of formal solutions (i.\,e.\ solutions in the form
of power series without regarding their convergence) for almost all initial
conditions.  In the sequel, we will always assume that in addition our
system is not underdetermined, i.\,e.\ that its formal solution space is
finite-dimensional.  Differentially simple systems can be obtained with the
differential Thomas decomposition.

Consider the ring of differential polynomials
$\mathcal{D}=\RR(t)\{\mathbf{u}\}$.  Obviously, the polynomials $p_{i}$ may
be considered as elements of $\mathcal{D}$ and we denote by
$\hat{\mathcal{I}}=\langle p_{1},\dots,p_{r}\rangle_{\mathcal{D}}$ the
differential ideal generated by the equations in our differentially simple
system.  It turns out that in some respect this ideal is too small and
therefore we saturate it with respect to the differential polynomial
$Q=\prod_{i=1}^{r}\mathrm{init}(p_{i})\mathrm{sep}(p_{i})$ to obtain the
differential ideal $\mathcal{I}=\hat{\mathcal{I}}:Q^{\infty}$ of which one
can show that it is the radical of $\hat{\mathcal{I}}$
\cite[Prop.~2.2.72]{dr:habil}.  Over the real numbers, we need the
potentially larger real radical according to the real nullstellensatz (see
e.\,g.\ \cite[Sect.~4.1]{bcr:realag} for a discussion).  An algorithm for
determining the real radical was proposed by Becker and Neuhaus
\cite{bn:realrad,rn:realrad2}.  An implementation over the rational numbers
exists in \textsc{Singular} \cite{ss:realrad}.  However, in all these
references it is assumed that one deals with an ideal in a polynomial ring
with finitely many variables.  Thus we have to postpone the determination
of the real radical until we have obtained such an ideal.

For the transition from differential algebra to jet geometry, we introduce
for any finite order $\ell\in\NN$ the finite-dimensional subrings
$\mathcal{D}_{\ell}=
\mathcal{D}\cap\RR[t,\mathbf{u},\dots,\mathbf{u}^{(\ell)}]$.  Note that
$\mathcal{D}_{\ell}$ is the coordinate ring of the jet bundle $J_{\ell}\pi$
considered as an affine space.  Fixing some order $\ell\in\NN$ which is at
least the maximal order of an equation~$p_{i}=0$ or an inequality
$q_{j}>0$, we define the polynomial ideal
$\hat{\mathcal{I}}_{\ell}=\hat{\mathcal{I}}\cap\mathcal{D}_{\ell}$.  Using
Janet--Riquier theory and Gr\"obner basis techniques, it is straightforward
to construct an explicit generating set of this ideal.  Now that we have an
ideal in a polynomial ring with finitely many variables, we can determine
its real radical $\mathcal{I}_{\ell}$.  Finally, we prefer to work with
irreducible sets and thus perform a real prime decomposition of the ideal
$\mathcal{I}_{\ell}$ and study each prime component
separately.\footnote{Over the complex numbers, one can show that the
  radical $\mathcal{I}_{\ell}$ obtained after the saturation with $Q$ is
  always equidimensional \cite[Thm.~1.94]{lh:phd} and therefore does not
  possess embedded primes.  It is unclear whether the real radical shares
  this property.  For our geometric analysis, it suffices to study only the
  minimal primes.} Thus we may assume in the sequel without loss of
generality that the given polynomials~$p_{i}$ generate directly a real
prime ideal $\mathcal{I}_{\ell}\subset\mathcal{D}_{\ell}$.

\begin{definition}\label{def:wellprep}
  A basic semialgebraic differential equation
  $\mathcal{J}_{\ell}\subset J_{\ell}\pi$ is called \emph{well prepared},
  if it is obtained by the above outlined procedure starting from a
  differentially simple system.
\end{definition}

Consider a (local) solution
$\mathbf{g}:\mathcal{I}\subseteq\RR\rightarrow\RR^{m}$ of a semialgebraic
differential equation $\mathcal{J}_{\ell}$ and the corresponding curve
$\gamma_{\mathbf{g}}:\mathcal{I}\rightarrow J_{\ell}\pi$ given by
$t\mapsto \bigl(t,\mathbf{g}(t), \mathbf{\dot{g}}(t), \dots,
\mathbf{g}^{(\ell)}(t)\bigr)$.  Since, according to our definition of a
solution, $\mathrm{im}\,\gamma_{\mathbf{g}}\subseteq\mathcal{J}_{\ell}$,
for each $t\in\mathcal{I}$ the tangent vector $\gamma_{\mathbf{g}}'(t)$
must lie in the tangent space
$T_{\gamma_{\mathbf{g}}(t)}\mathcal{J}_{\ell}$ of $\mathcal{J}_{\ell}$ at
the point $\gamma_{\mathbf{g}}(t)\in\mathcal{J}_{\ell}$.  We mentioned
already above the contact structure of the jet bundle.  It characterises
intrinsically those (transversal) curves
$\gamma:\mathcal{I}\subseteq\RR\rightarrow J_{\ell}\pi$ that are prolonged
graphs.  More precisely, there exists a function
$\mathbf{g}:\mathcal{I}\rightarrow\RR^{m}$ such that
$\gamma=\gamma_{\mathbf{g}}$, if and only if the tangent vector
$\gamma'(t)$ is contained in the contact distribution
$\mathcal{C}^{(\ell)}|_{\gamma(t)}$ evaluated at $\gamma(t)$.  These two
observations motivate the following definition of the space of all
``infinitesimal solutions'' of the differential equation
$\mathcal{J}_{\ell}$.

\begin{definition}\label{def:vess}
  Given a point $\rho$ on a semialgebraic jet set
  $\mathcal{J}_{\ell}\subseteq J_{\ell}\pi$, we define the \emph{Vessiot
    space} at $\rho$ as the linear space
  $\mathcal{V}_{\rho}[\mathcal{J}_{\ell}]=
  T_{\rho}\mathcal{J}_{\ell}\cap\mathcal{C}^{(\ell)}|_{\rho}$.
\end{definition}

In general, the properties of the Vessiot spaces
$\mathcal{V}_{\rho}[\mathcal{J}_{\ell}]$ depend on their base point $\rho$.
In particular, at different points the Vessiot spaces may have different
dimensions.  Nevertheless, it is easy to show that for a well-prepared
semialgebraic differential equation $\mathcal{J}_{\ell}$ the Vessiot spaces
define a smooth regular distribution on a Zariski open and dense subset of
$\mathcal{J}_{\ell}$ (see e.\,g.\ \cite[Prop.~2.10]{lrss:gsade} for a
rigorous proof).  Therefore, with only a minor abuse of language, we will
call the family of all Vessiot spaces the \emph{Vessiot distribution}
$\mathcal{V}[\mathcal{J}_{\ell}]$ of the given differential equation
$\mathcal{J}_{\ell}$.

We will ignore here \emph{algebraic singularities} of a semialgebraic
differential equation $\mathcal{J}_{\ell}$, i.\,e.\ points on
$\mathcal{J}_{\ell}$ that are singularities in the sense of algebraic
geometry.  It is a classical task in algebraic geometry to find them,
e.\,g.\ with the Jacobian criterion which reduces the problem to linear
algebra \cite[Thm.~9.6.9]{clo:iva}.  We will focus instead on
\emph{geometric singularities}.  In the here exclusively considered case of
not underdetermined ordinary differential equations, we can use the
following -- compared with \cite{lrss:gsade} simplified -- definition which
is equivalent to the classical definition given e.\,g.\ in
\cite{via:geoode}.

\begin{definition}\label{def:sing}
  Let $\mathcal{J}_{\ell}\subseteq J_{\ell}\pi$ be a well-prepared, not
  underdetermined, semialgebraic jet set.  A smooth point
  $\rho\in\mathcal{J}_{\ell}$ with Vessiot space
  $\mathcal{V}_{\rho}[\mathcal{J}_{\ell}]$ is called
  \begin{enumerate}[(i)]
  \item \emph{regular}, if $\dim{\mathcal{V}_{\rho}[\mathcal{J}_{\ell}]}=1$
    and
    $\mathcal{V}_{\rho}[\mathcal{J}_{\ell}]\cap\ker{T_{\rho}}\pi^{\ell}=0$,
  \item \emph{regular singular}, if
    $\dim{\mathcal{V}_{\rho}[\mathcal{J}_{\ell}]}=1$ and
    $\mathcal{V}_{\rho}[\mathcal{J}_{\ell}]\subseteq\ker{T_{\rho}}\pi^{\ell}$,
  \item \emph{irregular singular}, if
    $\dim{\mathcal{V}_{\rho}[\mathcal{J}_{\ell}]}>1$.
  \end{enumerate}
\end{definition}

Thus irregular singularities are characterised by a jump in the dimension
of the Vessiot space.  At a regular singularity, the Vessiot space
$\mathcal{V}_{\rho}[\mathcal{J}_{\ell}]$ has the ``right'' dimension,
i.\,e.\ the same as at a regular point, but in the ambient tangent space
$T_{\rho}J_{\ell}\pi$ its position relative to the subspace
$\ker{T_{\rho}}\pi^{\ell}$ is ``wrong'': it lies vertical, i.\,e.\ it is
contained in $\ker{T_{\rho}}\pi^{\ell}$.  By contrast, at regular points
the Vessiot space is $\pi^{\ell}$-transversal, since
$\mathcal{V}_{\rho}[\mathcal{J}_{\ell}]\cap\ker{T_{\rho}}\pi^{\ell}=0$.
The relevance of this distinction is that any tangent vector to the
prolonged graph of a function is always $\pi^{\ell}$-transversal.  Hence no
prolonged solution can go through a regular singularity.

A sufficiently small (Euclidean) neighbourhood of an arbitrary regular
point can be foliated by the prolonged graphs of solutions.  At a regular
singular point, there still exists a foliation of any sufficiently small
neighbourhood by integral curves of the Vessiot distribution.  However, at
such a point these curves can no longer be interpreted as prolonged graphs
of functions (see \cite{wms:aims} or \cite{wms:quasilin} for a more
detailed discussion).  The set of all regular and all regular singular
points is the above mentioned Zariski open and dense subset of
$\mathcal{J}_{\ell}$ on which the Vessiot spaces define a smooth regular
distribution.  At the irregular singular points, the classical uniqueness
results fail and it is possible that several (even infinitely many)
prolonged solutions are passing through such a point.

\section{Parametric Gaussian Elimination}

We will show in the next section that an algorithmic realisation of
Definition \ref{def:sing} essentially boils down to analysing a parametric
linear system of equations. Therefore we study now \emph{parametric
  Gaussian elimination} in some detail and propose a corresponding
algorithm that satisfies a number of particular requirements coming with
our application to differential equations.  While parametric Gaussian
elimination has beed studied in theory and practice for more than 30 years,
e.\,g.~\cite{BallarinKauers:04a,Grigoriev:88a,Sit:92a}, it is still not
widely available in contemporary computer algebra systems. One reason might
be that it calls for logic and decision procedures for an efficient
heuristic processing of the potentially exponential number of cases to be
considered. The algorithm proposed here is based on experiences with the
\textsc{PGauss} package which was developed in \textsc{Reduce}
\cite{Hearn:67a,Hearn:05a} as an unpublished student's project under
co-supervision of the third author in 1998. The original motivation at that
time was the investigation of possible integration and implicit use of the
\textsc{Reduce} package \textsc{Redlog} for interpreted first-order logic
\cite{DolzmannSturm:97a,Sturm:06a,Sturm:07a} in core domains of computer
algebra (see also \cite{DolzmannSturm:97b}).

For our proof-of-concept purposes here, we keep the algorithm quite basic
from a linear algebra point of view.  For instance, we do not perform
Bareiss division \cite{Bareiss:68a}, which is crucial for polynomial
complexity bounds in the non-parametric case.  On the other hand, we apply
strong heuristic simplification techniques \cite{DolzmannSturm:97c} and
quantifier elimination-based decision procedures
\cite{Kosta:16a,Seidl:06a,Weispfenning:88a,Weispfenning:97b} from
\textsc{Redlog} for pruning at an early stage the potentially exponential
number of cases to be considered.

In a rigorous mathematical language, we consider the following problem over
a field $K$ of characteristic $0$.  We are given an $M\times N$ matrix $A$
with entries from a polynomial ring $\ZZ[\vv]$ whose $P$ variables
$\vv=(v_{1},\dots,v_{P})$ are considered as parameters.  In dependence of
the parameters $\vv$, we are interested in determining the solution space
$S\subseteq K^{N}$ of the homogeneous linear system $A\xx=0$ in the
unknowns $\xx=(x_{1},\dots,x_{N})$.  Furthermore, we assume that we are
given a sublist $\yy\subseteq\xx$ of unknowns defining the linear subspace
$\Pi_\yy(K^N) := \{\,\xx \in K^N \mid \text{$x_i = 0$ for
  $x_i \in \yy$}\,\} \subseteq K^{N}$ and we also want to determine the
dimension of the intersection $S \cap \Pi_\yy(K^N)$.  A parametric Gaussian
elimination is for us then a procedure that produces from these data a list
of pairs $(\Gamma_{i}, \Eta_{i})_{i=1,\dots,I}$.  Each \emph{guard}
$\Gamma_{i}$ describes a semialgebraic subset
$G(\Gamma_{i}) = \{\,\bar\vv \in K^P \mid K, (\vv=\bar\vv) \models
\Gamma_{i}\,\}$ of the parameter space $K^{P}$.  The respective
\emph{parametric solution} $\Eta_{i}$ represents the solution space
$S(\Eta_{i})$ of $A\xx = 0$ for all parameter values
$\bar\vv \in G(\Gamma_{i})$ in the following sense.

\begin{definition}
  Let $A \in \ZZ[\vv]^{M \times N}$, and let $\bar\vv \in K^P$ be some
  parameter values.  A \emph{parametric solution of $A\xx = 0$ suitable for
    $\bar\vv$} is a set of formal equations
  \begin{displaymath}
    \Eta=\{
    x_{\perm{1}} = s_{1},\,
    \dots,\,
    x_{\perm{L}} = s_{L},\,
    x_{\perm{L+1}} = r_{N-L},\,
    \dots,\,
    x_{\perm{N}} = r_{1}
    \},
  \end{displaymath}
  where $L \in \{1, \dots, N\}$, $\pi$ is a permutation of
  $\{1, \dots, N\}$, we have
  $s_{n} \in \ZZ(\vv,x_{\perm{n+1}}, \dots, x_{\perm{N}})$ for
  ${n \in \{1, \dots, L\}}$, and $r_1$, \dots, $r_{N-L}$ are new
  indeterminates.  We call $x_{\perm{L+1}}$, \dots,~$x_{\perm{N}}$
  \emph{independent variables}.\footnote{The introduction of the new
    indeterminates $r_1$, \dots,~$r_{N-L}$ is somewhat redundant.  Our
    motivation is to mimic the output of \textsc{Reduce}, which uses at
    their place operators \texttt{arbreal(n)} or \texttt{arbcomplex(n)},
    respectively.} If one substitutes $\vv = \bar\vv$, then the following
  holds.  The denominator of any rational function $s_n$ does not vanish.
  For an arbitrary choice of values $\bar r_1$,
  \dots,~$\bar r_{N-L} \in K$, one obtains values $\bar s_1$,
  \dots,~$\bar s_L \in K$ such that
  \begin{displaymath}
    \bar x_{\perm{1}} = \bar s_{1},\quad
    \dots,\quad
    \bar x_{\perm{L}} = \bar s_{L},\quad
    \bar x_{\perm{L+1}} = \bar r_{N-L},\quad
    \dots,\quad
    \bar x_{\perm{N}} = r_{1}
  \end{displaymath}
  defines a solution $\bar\xx\in K^{N}$ of $A\xx = 0$.  Vice versa, every
  solution $\bar\xx\in K^{N}$ of $A\xx = 0$ can be obtained this way for
  some choice of values $\bar r_1$, \dots,~$\bar r_{N-L} \in K$.
\end{definition}

In addition, we require that
$\dim{\bigl(S(\Eta_{i}) \cap \Pi_\yy(K^N)\bigr)}$ is constant on the set
$G(\Gamma_{i})$ and that $G(\Gamma_{i}) \cap G(\Gamma_{j}) = \emptyset$ for
$i \neq j$ and $\bigcup_{i=1}^{I}G(\Gamma_{i}) = K^{P}$, i.\,e.\ that the
guards provide a disjoint partitioning of the parameter space.

Our Gauss algorithm will use a logical deduction procedure $\vdash_K$ to
derive from conditions $\Gamma$ whether or not certain matrix entries
vanish in $K$. The correctness of our algorithm will require only two very
natural assumptions on $\vdash_K$:
\begin{enumerate}[$D_1$.]
\item $\Gamma \vdash_K \gamma$ implies
  $K, \Gamma \models \gamma$, i.e., $\vdash_K$ is sound;
\item $\gamma \land \Gamma \vdash_K \gamma$, i.e., $\vdash_K$ can derive
  constraints that literally occur in the premise.
\end{enumerate}
Of course, our notation in $D_2$ should to be read modulo associativity and
commutativity of the conjunction operator.  Notice that $D_2$ is easy to
implement, and implementing only $D_2$ is certainly sound.
Algorithm~\ref{alg:pgauss} describes then our parametric Gaussian
elimination.

\begin{algorithm}[h]
  \caption{ParametricGauss}\label{alg:pgauss}
  \begin{algorithmic}[1]
    \REQUIRE Denote $\vv = (v_1, \dots, v_P)$, $\xx = (x_1, \dots, x_N)$:
    \begin{enumerate}[(i)]
    \item matrix $A\in\ZZ[\vv]^{M\times N}$
    \item list $\xx$
    \item sublist $\yy$ of $\xx$
    \item field $K$ of characteristic $0$ with a suitable deduction
      procedure $\vdash_K$ 
    \end{enumerate}
    
    \ENSURE list $(\Gamma_{i}, \Eta_{i})_{i=1,\dots,I}$ as follows:
    \begin{enumerate}[(i)]
    \item each $\Gamma_{i}$ is a conjunction of polynomial equations and
      inequations in variables $\vv$
    \item given $\bar \vv \in K^P$, we have $\bar\vv \in G(\Gamma_i)$ for one and only
      one \emph{matching case} $i \in \{1, \dots, I\}$
    \item given $\bar \vv \in K^P$ with unique matching case $i$, $\Eta_i$ is a
      solution of $A\xx=0$ suitable for $\bar\vv$
    \item $\dim\bigl(S(\Eta_{i}) \cap \Pi_\yy(K^N)\bigr)$ is constant on $G(\Gamma_{i})$
    \end{enumerate}
    \medskip
    
    \STATE $Y:=\{\,n\in\{1,\dots,N\}\mid \text{$x_n$ in $\yy$}\,\}$
    \STATE $I:=0$
    \STATE create an empty stack
    \STATE
    $\operatorname{push}~(\true, A, 1)$
    \WHILE{stack is not empty}
    \STATE $(\Gamma, A, p):=\operatorname{pop}$
    \IF{$\Gamma\nvdash_K\false$}
    %
    %
    \IF{there is $m\in\{p,\dots,M\}\setminus Y$, $n\in\{p,\dots,N\}$ such that
      $\Gamma\vdash_K A_{mn}\neq0$}
    \STATE in $A$, swap rows $p$ with $m$ and columns $p$ with $n$
    \STATE in $A$, use row $p$ to obtain $A_{p+1,p}= \dots = A_{m,p}=0$
    \STATE $\operatorname{push}~(\Gamma, A, p+1)$
    \ELSIF{there is $m\in\{p,\dots,M\}\setminus Y$, $n\in\{p,\dots,N\}$ such
      that $\Gamma\nvdash_K A_{mn}=0$}
    \STATE $\operatorname{push}~(\Gamma\land A_{mn}\neq0, A, p)$
    \STATE in $A$, set $A_{mn}:=0$ \{this is an optional optimisation of the
    Algorithm\}
    \STATE $\operatorname{push}~(\Gamma\land A_{mn}=0, A, p)$
    \ELSIF{there is $m\in\{p,\dots,M\}\cap Y$, $n\in\{p,\dots,N\}$ such that
      $\Gamma\vdash_K A_{mn}\neq0$}
    %
    %
    \STATE in $A$, swap rows $p$ with $m$ and columns $p$ with $n$
    \STATE in $A$, use row $p$ to obtain $A_{p+1,p}= \dots = A_{m,p}=0$
    \STATE $\operatorname{push}~(\Gamma, A, p+1)$
    \ELSIF{there is $m\in\{p,\dots,M\}\cap Y$, $n\in\{p,\dots,N\}$ such that
      $\Gamma\nvdash_K A_{mn}=0$}
    %
    %
    \STATE $\operatorname{push}~(\Gamma \land A_{mn}\neq0, A, p)$
    \STATE in $A$, set $A_{mn}:=0$ \{this is an optional optimisation of the
    Algorithm\}
    \STATE $\operatorname{push}~(\Gamma \land A_{mn}=0, A, p)$
    \ELSE[$A$ is in row echelon form modulo $\Gamma$]
    \STATE
    $I:=I+1$
    \STATE $(\Gamma_I, \Eta_I) := (\Gamma, \text{construct $\Eta_I$ from $A$})$
    \ENDIF
    %
    \ENDIF
    \ENDWHILE
    \RETURN $(\Gamma_i,\Eta_i)_{i=1,\dots,I}$
  \end{algorithmic}
\end{algorithm}

\begin{proposition}\label{prop:termination}
  Algorithm~\ref{alg:pgauss} terminates.
\end{proposition}

\begin{proof}
  For each possible stack element $s = (\Gamma, A, p)$ define
  \begin{eqnarray*}
    \mu_1(s) &=& \min{\{M, N\}} - p \in \NN,\\
    \mu_2(s) &=& |\{\,(m,n) \in \{p,\dots,M\} \times \{p,\dots,N\} :
               \text{$\Gamma \nvdash A_{mn} \neq 0$ and $\Gamma \nvdash  A_{mn} = 0$}\,\}| \in \NN.
  \end{eqnarray*}
  During execution, we associate with the current stack a multiset
  \begin{displaymath}
    \mu(S)=\{\,(\mu_1(s), \mu_2(s))\in\NN^2\mid s\in S\,\}.
  \end{displaymath}
  Every execution of the while-loop removes from $\mu(S)$
  exactly one pair and adds to $\mu(S)$ at most finitely many pairs, all of
  which are lexicographically smaller than the removed one. This guarantees
  termination, because the corresponding multiset order is well-founded
  \cite{BaaderNipkow:98a}.
\end{proof}

It is obvious that the output of Algorithm~\ref{alg:pgauss} satisfies
property (i) of its specification from the way the guards $\Gamma_{i}$ are
constructed.  The same is true for property (iii), as
Algorithm~\ref{alg:pgauss} determines for each arising case a row echelon
form where the guard $\Gamma_{i}$ ensures that all pivots are non-vanishing
on $G(\Gamma_{i})$.  Finally, property (iv) is a consequence of our
pivoting strategy: pivots in $\yy$-columns are chosen only when all
remaining $\xx$-columns contain only zeros in their relevant part.  Hence
Algorithm~\ref{alg:pgauss} produces a row echelon form where rows with a
pivot in a $\yy$-column can only occur in the bottom rows after all the
rows with pivots in $\xx$-columns.  As a by-product, our pivoting strategy
has the effect that the algorithm prefers the variables in $\yy$ over the
remaining variables when it chooses the independent variables.  The next
proposition proves property (ii) and thus the correctness of
Algorithm~\ref{alg:pgauss}.  We remark that Ballarin and Kauers
\cite[Section~5.3]{BallarinKauers:04a} observed that the well-known
approach taken by Sit \cite{Sit:92a} does not have this property which is
crucial for our application of parametric Gaussian elimination in the
context of differential equations.

\begin{proposition}\label{prop:cd}
  Let $(\Gamma_i, \Eta_i)_{i=1,\dots,I}$ be an output obtained from
  Algorithm~\ref{alg:pgauss}. Then
  \begin{displaymath}
    G(\Gamma_i) \cap G(\Gamma_j) = \emptyset \quad (i \neq j), \qquad
    \bigcup_{i=1}^IG(\Gamma_i) = K^P.
  \end{displaymath}
  In other words, given $\bar\vv \in K^P$, there is one and only one
  $i \in \{1, \dots, I\}$ such that $K, (\vv = \bar\vv) \models \Gamma_i$.
\end{proposition}

\begin{proof}
  We consider a run of Algorithm~\ref{alg:pgauss} with output
  $(\Gamma_i, \Eta_i)_{i=1,\dots,I}$. We observe the state $\QQQ_k$ of the algorithm
  right before the $k$th iteration of the test for an empty stack in line 5: Let
  $\QQQ_k = \SSS_k \cup \RRR_k$ where ${\SSS_k = \{\,\Gamma \mid
    \text{$(\Gamma, A, p)$ on the stack for some $A$, $p$}\,\}}$ and
  $\RRR_k = \{\Gamma_1, \dots, \Gamma_I\}$. Line 5 is executed at least once and, by
  Proposition~\ref{prop:termination}, only finitely often, say $\ell$ times. The
  $\ell$th test fails with an empty stack, $\SSS_\ell = \emptyset$, and
  $\QQQ_\ell = \RRR_\ell = \{\Gamma_1, \dots, \Gamma_I\}$ contains the guards of the output. It now
  suffices to show the following invariants of $\QQQ_k$:
  \begin{enumerate}[${I}_1.$]
  \item $G(\Gamma) \cap G(\Gamma') = \emptyset$ for $\Gamma$, $\Gamma' \in \QQQ_k$
    with $\Gamma \neq \Gamma'$,
  \item $\bigcup_{\Gamma\in \QQQ_k}G(\Gamma)=K^P$.
  \end{enumerate}
  The initialisations in lines 2 and 4 yield $\QQQ_1=\{\true\}$, which satisfies
  both $I_1$ and $I_2$. Assume now that $\QQQ_k$ satisfies $I_1$ and $I_2$, and
  consider $\QQQ_{k+1}$. In line 6, $\Gamma$ is removed from
  $\SSS_k \subseteq \QQQ_k$. Afterwards one and only one of the following cases applies:
  \begin{enumerate}[(a)]
  \item The if-condition in line 8 holds: Then
    $\QQQ_{k+1} = \bigl((\SSS_k \setminus \{\Gamma\}) \cup \{\Gamma\}\bigr) \cup \RRR_k = \QQQ_k$.
  \item The if-condition in line 12 holds: Then
    \begin{displaymath}
      \QQQ_{k+1}=\bigl((\SSS_k \setminus \{\Gamma\}) \cup \{\Gamma \land A_{mn} \neq 0, \Gamma \land A_{mn} =
      0\}\bigr) \cup \RRR_k.
    \end{displaymath}
    To show $I_1$, consider $\Gamma \land A_{mn} \neq 0 \in \QQQ_{k+1}$, and let
    $\Gamma' \in \QQQ_{k+1}$ with
    $\Gamma' \dotneq (\Gamma \land A_{mn} \neq 0)$. Using $I_1$ for $\QQQ_k$, we obtain
    \begin{displaymath}
      G(\Gamma \land A_{mn} \neq 0) \cap G(\Gamma') \subseteq G(\Gamma) \cap G(\Gamma') \dotequal \emptyset.
    \end{displaymath}
    The same argument holds for $\Gamma \land A_{mn} = 0 \in \QQQ_{k+1}$. To show
    $I_2$, we use $I_1$ for $\QQQ_{k+1}$ and $I_2$ for $\QQQ_k$ to obtain
    \begin{align*}
      \bigcup_{\Delta \in \QQQ_{k+1}}G(\Delta)
      \dotequal \bigcup_{\Delta \in \QQQ_k\atop \Delta \neq \Gamma} G(\Delta)
         \cup G(\Gamma \land A_{mn} \neq 0)
         \cup G(\Gamma \land A_{mn} = 0)
      \dotequal \bigcup_{\Delta \in \QQQ_{k}} G(\Delta)
      \dotequal K^P.
    \end{align*}
  \item The if-condition in line 16 holds: Then lines 17--19 are identical to
    lines 9--11, and we proceed as in case (a).
  \item The if-condition in line 20 holds: Then lines 21--23 are identical to
    lines 13--15, and we proceed as in case (b).
  \item We reach line 26 in the else-case: Then
    $\QQQ_{k+1} = (\SSS_k \setminus \{\Gamma\}) \cup (\RRR_k \cup \{\Gamma\}) = \QQQ_k$.\qedhere
  \end{enumerate}
\end{proof}


Inspection of the proofs yields that Proposition~\ref{prop:termination}
relies on properties $D_1$ and $D_2$ of our deduction $\vdash_K$ but
remains correct also with stronger sound deductions.
Proposition~\ref{prop:cd} does not refer to $\vdash_K$ except for the
termination result in Proposition~\ref{prop:termination}.  This paves the
way for the application of heuristic simplification techniques during
deduction, which we will discuss in more detail in
Section~\ref{sec:computations}.

\section{Detecting Geometric Singularities with Logic}

The main point of this article is an algorithmic realisation of Definition
\ref{def:sing}.  Obviously, as a first step one must be able to compute the
Vessiot space $\mathcal{V}_{\rho}[\mathcal{J}_{\ell}]$ at a point
$\rho\in\mathcal{J}_{\ell}$.  As we are only interested in smooth points,
this requires only some linear algebra.  We choose as ansatz for
constructing a vector $\mathbf{v}\in\mathcal{V}_{\rho}[\mathcal{J}_{\ell}]$
a general element
$\mathbf{v}=a C^{(\ell)}_{\mathrm{trans}}+
\sum_{\alpha=1}^{m}b_{\alpha}C^{(\ell)}_{\alpha}$ of the contact space
$\mathcal{C}^{(\ell)}|_{\rho}$ where $a,\mathbf{b}$ are yet undetermined
real coefficients.  We have
$\mathbf{v}\in\mathcal{V}_{\rho}[\mathcal{J}_{\ell}]$, if and only if
$\mathbf{v}$ is tangential to $\mathcal{J}_{\ell}$.

Recall that we always assume that our semialgebraic differential equation
$\mathcal{J}_{\ell}$ is given explicitly as a finite union of basic
semialgebraic differential equations each of which is well prepared.
Furthermore, $\rho$ is a smooth point of $\mathcal{J}_{\ell}$.  Thus, if
$\rho$ is contained in several basic semialgebraic differential equations,
then the equations parts of the corresponding systems must be equivalent in
the sense that they describe the same variety.  As we will see, in this
case we can choose for the subsequent analysis any of these basic
semialgebraic differential equations; the results will be independent of
this choice.

Without loss of generality, we may therefore assume that
$\mathcal{J}_{\ell}$ is actually a basic semialgebraic differential
equation described by a basic semialgebraic system with equations $p_{i}=0$
for $1\leq i\leq r$.  By a classical result in differential geometry (see
e.\,g.\ \cite[Prop.~1.35]{olv:lgde} for a simple proof), the
vector~$\mathbf{v}$ is tangential to $\mathcal{J}_{\ell}$, if and only if
$\mathbf{v}(p_{i})=0$ for all $i$.  Hence, we obtain the following
homogeneous linear system of equations for the unknowns $a,\mathbf{b}$ in
our ansatz:
\begin{equation}\label{eq:vess}
  C^{(\ell)}_{\mathrm{trans}}(p_{i})|_{\rho}a +
  \sum_{\alpha=1}^{m}C^{(\ell)}_{\alpha}(p_{i})|_{\rho}b_{\alpha}=0\,,\qquad
  i=1,\dots,r\,.
\end{equation}
At any \emph{fixed} point $\rho\in\mathcal{J}_{\ell}$, \eqref{eq:vess}
represents a linear system with real coefficients which is elementary to
solve.  The conditions for the various cases in Definition \ref{def:sing}
can now be interpreted as follows.  A point is an irregular singularity, if
and only if the dimension of the solution space of \eqref{eq:vess} is
greater than one.  At a regular point, the one-dimensional solution space
must have a trivial intersection with $\ker{T_{\rho}\pi^{\ell}}$, i.\,e.\
be $\pi^{\ell}$-transversal.  As in our ansatz only the vector
$C^{(\ell)}_{\mathrm{trans}}$ is $\pi^{\ell}$-transversal, this is the case
if and only if we have $a\neq0$ for all nontrivial solutions of
\eqref{eq:vess}.  Expressing these considerations via the rank of the
coefficient matrix of \eqref{eq:vess} and of the submatrix obtained by
dropping the column corresponding to the unknown $a$, we arrive at the
following statement.

\begin{proposition}\label{prop:sing}
  The point $\rho\in\mathcal{J}_{\ell}$ is regular, if and only if the rank
  of the matrix $A$ with entries
  $A_{i\alpha}=C^{(\ell)}_{\alpha}(p_{i})|_{\rho}$ is $m$.  The point
  $\rho$ is regular singular, if and only if it is not regular and the rank
  of the augmented matrix
  $\Bigl( C^{(\ell)}_{\mathrm{trans}}(p_{i})|_{\rho} \mid A \Bigr)$ is $m$.
  In all other cases, $\rho$ is an irregular singularity.
\end{proposition}

\begin{remark}\label{rem:sing}
  The rigorous definition of a (not) underdetermined differential equation
  is rather technical and usually only given for regular equations without
  singularities (see e.\,g.\ \cite[Def.~7.5.6]{wms:invol}).  In the case of
  ordinary differential equations, it is straightforward to extend the
  definition to our more general situation: a basic semialgebraic
  differential equation $\mathcal{J}_{\ell}$ is not underdetermined, if and
  only if at almost all points $\rho\in\mathcal{J}_{\ell}$ the rank of the
  matrix $A$ (the so-called symbol matrix) defined in the above proposition
  is $m$.  Thus a generic point is regular, as it should be.  The geometric
  singularities form a semialgebraic set of lower dimension.
\end{remark}

\begin{example}\label{ex:sphere}
  We consider the first-order algebraic differential equation
  $\mathcal{J}_{1}\subset J_{1}\pi$ given by
  \begin{equation}
    \dot{u}^{2}+u^{2}+t^{2}-1=0\,.
  \end{equation}
  Geometrically, it corresponds to the two-dimensional unit sphere in the
  three-dimensional first-order jet bundle $J_{1}\pi$ for $m=1$ and can be
  easily analysed by hand.  The linear system \eqref{eq:vess} for the
  Vessiot spaces consists here only of one equation
  \begin{displaymath}
    (t+u\dot{u})a+\dot{u}b=0
  \end{displaymath}
  for two unknowns $a$ and $b$. The matrix $A$ introduced in
  Proposition~\ref{prop:sing} consists simply of the coefficient of $b$.
  Thus geometric singularities are characterised by the vanishing of this
  coefficient and hence form the equator $\dot{u}=0$ of the sphere.  Only
  two points on it are irregular singularities, namely $(0,\pm1,0)$, as
  there also the coefficient of $a$ vanishes and hence even the rank of the
  augmented matrix drops. All the other points on the equator are regular
  singular. In Figure~\ref{fig:sphere}, the regular singular points are
  shown in red and the two irregular singularities in yellow.  The figure
  also shows integral curves of the Vessiot distribution.  As one can see,
  they spiral into the irregular singularities and cross frequently the
  equator.  At each crossing their projections to the $t$-$u$ space change
  direction and hence they cannot be the graph of a function there. But
  between two crossings, the integral curves correspond to the graphs of
  prolonged solutions of the equation.
\end{example}

\begin{figure}
  \centering
  \includegraphics[width=0.5\textwidth]{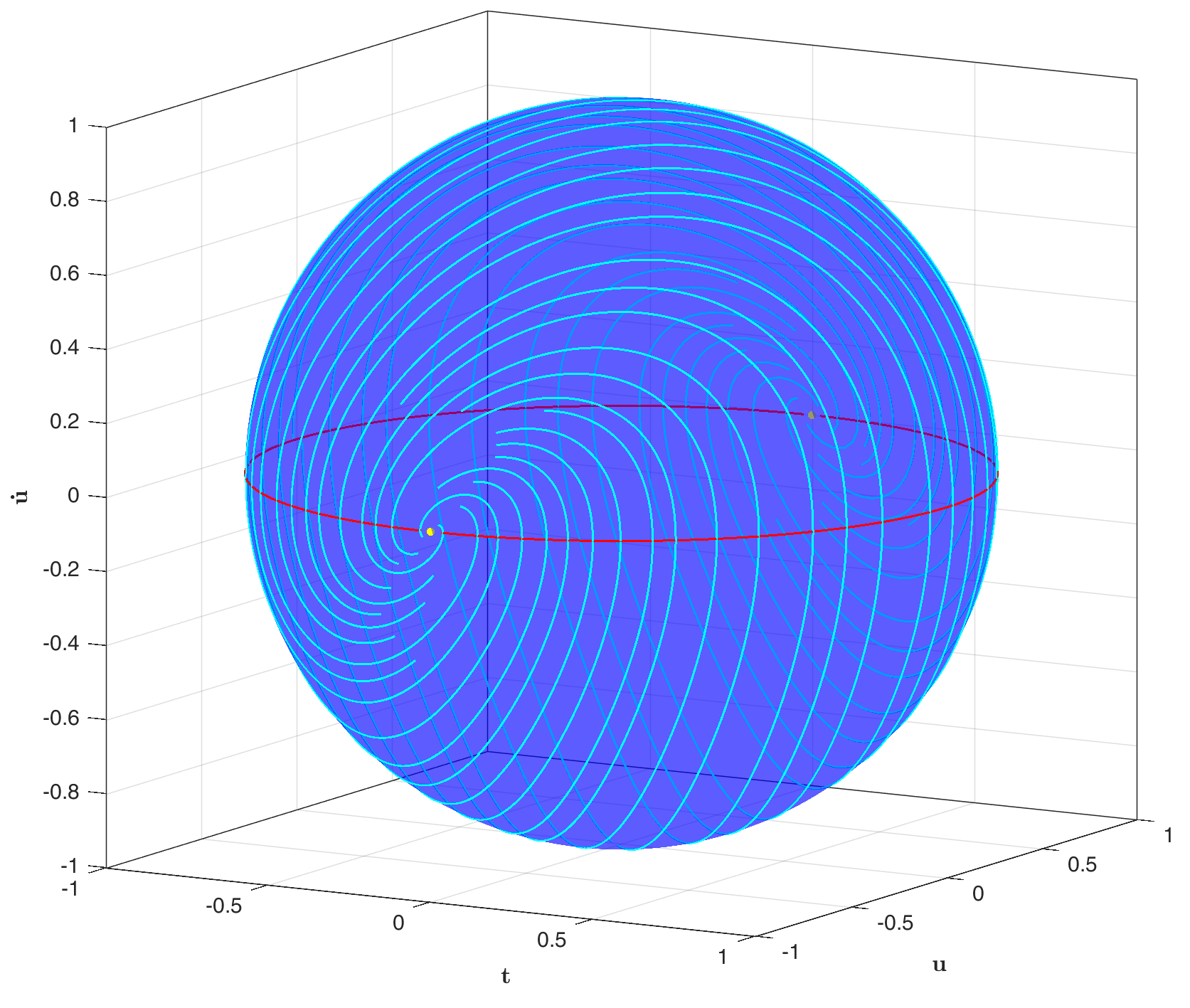}
  \caption{Unit sphere as semialgebraic differential equation}
  \label{fig:sphere}
\end{figure}
  
For systems containing equations of different orders or for systems
obtained by prolongations, the following observation (which may be
considered as a variation of \cite[Prop.~9.5.10]{wms:invol}) is useful, as
it significantly reduces the size of the linear system \eqref{eq:vess}.  It
requires that the system is well prepared, as it crucially depends on the
fact that no hidden integrability conditions are present.

\begin{proposition}\label{prop:prolvess}
  Let $\mathcal{J}_{\ell}\subset J_{\ell}\pi$ be a well-prepared basic
  semialgebraic differential equation of order~$\ell$.  Then it suffices to
  consider in the linear system \eqref{eq:vess} only those equations
  $p_{i}=0$ which are of order $\ell$; all other equations contribute only
  zero rows.
\end{proposition}

\begin{proof}
  By a slight abuse of notation (more precisely, by omitting some
  pull-backs), we have the following relation between the generators of the
  contact distributions of two neighbouring orders:
  \begin{displaymath}
    C^{(k+1)}_{\mathrm{trans}} = C^{(k)}_{\mathrm{trans}} +
        \sum_{\alpha=1}^{m}u_{\alpha}^{(k+1)}C^{(k)}_{\alpha}\,.
  \end{displaymath}
  On the other hand, if $\varphi$ is any function (not necessarily
  polynomial) depending only on jet variables up to an order $k<\ell$, then
  its formal derivative is given by
  $D\varphi=C^{(k+1)}_{\mathrm{trans}}(\varphi)$.  Since we assume that
  $\mathcal{J}_{\ell}$ is well prepared, for any equation $p_{i}=0$ in the
  corresponding basic semialgebraic system of order $k<\ell$ the prolonged
  equation $Dp_{i}=0$ can be expressed as a linear combination of the
  equations contained in the system (otherwise we would have found a hidden
  integrability condition).  Because of $k<\ell$, we have
  $Dp_{i}=C^{(k+1)}_{\mathrm{trans}}(p_{i})=C^{(\ell)}_{\mathrm{trans}}(p_{i})$
  and trivially $C^{(\ell)}_{\alpha}(p_{i})=0$ for all $\alpha$.  Hence the
  row contributed by $p_{i}$ to \eqref{eq:vess} is a zero row, as
  $Dp_{i}(\rho)=0$ at any point $\rho\in\mathcal{J}_{\ell}$.
\end{proof}

For the purpose of detecting all geometric singularities in a given
semialgebraic differential equation $\mathcal{J}_{\ell}$, we must analyse
the behaviour of \eqref{eq:vess} \emph{in dependency of the point $\rho$}.
Thus we must now consider the coefficients of \eqref{eq:vess} as
polynomials in the jet variables
$(t,\mathbf{u},\mathbf{\dot{u}},\dots,\mathbf{u}^{(\ell)})$ and not as real
numbers.  Furthermore, we must augment \eqref{eq:vess} by the semialgebraic
differential system defining $\mathcal{J}_{\ell}$ and study the combined
system of equations and inequalities in the variables
$(t,\mathbf{u},\mathbf{\dot{u}}\dots,\mathbf{u}^{(\ell)},a,\mathbf{b})$. In
the approach of \cite{lrss:gsade}, one simply performs an algebraic Thomas
decomposition of this system for a suitable ranking of the variables.
While this approach is correct and identifies all geometric singularities,
it has some shortcomings.  It does not really exploit that a part of the
problem is linear and as it implicitly also determines an algebraic Thomas
decomposition of the differential equation~$\mathcal{J}_{\ell}$, it leads
in general to many redundant case distinctions, which are unnecessary for
solely detecting all real singularities, but simply reflect certain
geometric properties of the semialgebraic set $\mathcal{J}_{\ell}$.

We propose now as a novel approach to study the linear part \eqref{eq:vess}
separately from the underlying semialgebraic differential equation
$\mathcal{J}_{\ell}$ considering it as a \emph{parametric} linear system in
the unknowns $a$, $\mathbf{b}$ with the jet variables
$(t,\mathbf{u},\mathbf{\dot{u}}\dots,\mathbf{u}^{(\ell)})$ as (yet
independent) parameters.  Using parametric Gaussian elimination, all
possible different cases for the linear system are identified.  Then, in a
second step, it is verified for each case whether it occurs somewhere on
the differential equation~$\mathcal{J}_{\ell}$, i.\,e.\ we take now into
account that our parameters are not really independent but have to satisfy
a basic semialgebraic system.  If yes, we obtain by simply combining the
equations and inequalities describing the case distinction with the
equations and inequalities defining $\mathcal{J}_{\ell}$ a semialgebraic
description of the corresponding subset of $\mathcal{J}_{\ell}$.

According to Proposition~\ref{prop:sing}, the coefficient matrix $A$ of the
linear system \eqref{eq:vess} possesses the same rank at regular and at
regular singular points.  The difference between the two cases is the
relative position of the Vessiot space to the linear subspace
$W=\ker{T}_{\rho}\pi^{\ell}$: as one can see in Definition~\ref{def:sing},
at regular singular points the solution space lies in $W$, whereas at
regular points its intersection with $W$ is trivial.  For this reason, we
need a parametric Gaussian elimination in the form developed in the
previous section which takes the relative position of the solution space to
a prescribed linear (cartesian) subspace into account.  In terms of the
$m+1$ coefficients $a$, $\mathbf{b}$ of our ansatz, $W$ corresponds to the
cartesian subspace of $\RR^{m+1}$ defined by the equation $a=0$ (which we
can write as $\Pi_{a}(\RR^{m+1})$ in the notation of the last section) and
thus we solve \eqref{eq:vess} using Algorithm~\ref{alg:pgauss} with the
choice $\yy=(a)$.  This means that -- among the points with a
one-dimensional solution space -- we characterise the regular points as
those where $a$ is the free variable in our solution representation and the
regular singular points as those where $a=0$, i.\,e.\ where the
intersection of the solution space of \eqref{eq:vess} with
$\Pi_{a}(\RR^{m+1})$ is trivial.

Because of our special form of parametric Gaussian elimination and the
choice of $\yy=(a)$, all points on one of the obtained subsets
$G(\Gamma_{i})$ are of the same type in the sense of
Definition~\ref{def:sing}.  The type is easy to decide on the basis of the
form of the obtained row echelon form of the linear system (or of its
solution) on the subset.  Hence we do actually more than just detecting
singularities: we identify semialgebraic subsets of $\mathcal{J}_{\ell}$ on
which the Vessiot spaces allow for a uniform description and possess
uniform properties.  This is of great importance for a possible further
analysis of the found singularities (not discussed here).

In a more formal language, our novel approach translates into Algorithm
\ref{alg:realsing}, the correctness of which follows from the above
discussion.  Note that for computational purposes we limit ourselves to
input with integer coefficients.  The critical steps are the parametric
Gaussian elimination which may potentially lead to many case distinctions,
but which represents otherwise a linear operation.  For each obtained case,
an existential closure must be studied to check whether the case actually
occurs on $\mathcal{J}_{\ell}$.  The real quantifier elimination in
\textsc{Redlog} primarily uses virtual substitution techniques
\cite{Kosta:16a,Sturm:17a,Sturm:18a,Weispfenning:88a,Weispfenning:97b} and
falls back into partial cylindrical algebraic decomposition
\cite{Collins:75,CollinsHong:91,Seidl:06a} for subproblems where degree
bounds are exceeded.  The latter algorithm is double exponential in the
worst case \cite{Brown:2007:CQE:1277548.1277557}.  It is noteworthy that
for our special case of existential sentences also single exponential
algorithms exist \cite{Grigoriev:88a} but no corresponding implementations.

\begin{algorithm}
  \caption{RealSingularities}\label{alg:realsing}
  \begin{algorithmic}[1]
    \REQUIRE well-prepared, basic semialgebraic system
    $\Sigma_{\ell}=\bigl((p_{a}=0)_{a=1,\dots,A},
    (q_{b}>0)_{b=1,\dots,B}\bigr)$,
    where $p_a$, $q_b\in\mathcal{D}_{\ell}=
    \mathcal{D}\cap\ZZ[t,\mathbf{u},\dots,\mathbf{u}^{(\ell)}]$
    \ENSURE finite system $(\Gamma_{i},\Eta_{i})_{i=1,\dots,I}$ with
        \begin{enumerate}[(i)]
        \item each $\Gamma_{i}$ is a
          disjunctive normal form of polynomial equations,
          inequations, and inequalities over $\mathcal{D}_{\ell}$
          describing a semialgebraic subset
          $\mathcal{J}_{\ell,i}\subseteq\mathcal{J}_{\ell}$ 
        \item each $\Eta_{i}$ describes the Vessiot spaces of all points on
          $\mathcal{J}_{\ell,i}$
        \item all sets $\mathcal{J}_{\ell,i}$ are disjoint and their union
          is $\mathcal{J}_{\ell}$
        \end{enumerate}
    \STATE set up the matrix $A$ of the linear system \eqref{eq:vess} using
    the equations $(p_{a}=0)_{a=1,\dots,A}$
    \STATE  $\Pi=\bigl(\gamma_{\tau},\Eta_{\tau}\bigr)_{\tau=1,\dots,t}:=
        \mathtt{ParametricGauss}\bigl(A,(\mathbf{b},a),(a),\RR\bigr)$   
    \FOR{$\tau:=1,\dots,t$}
    \STATE let $\Gamma_{\tau}$ be a disjunctive normal form of
    $\gamma_{\tau}\land\bigwedge\Sigma_{\ell}$
    \STATE check satisfiability of $\Gamma_{\tau}$ using real quantifier
    elimination on 
    $\exists t\,\exists\mathbf{u}\dots\exists\mathbf{u}^{(\ell)}\,
    \Gamma_{\tau}$
        \IF{$\Gamma_{\tau}$ is unsatisfiable}
            \STATE delete $(\gamma_{\tau},\Eta_{\tau})$ from $\Pi$
        \ELSE
            \STATE replace $(\gamma_{\tau},\Eta_{\tau})$  by
            $(\Gamma_{\tau},\Eta_{\tau})$ in $\Pi$
        \ENDIF
    \ENDFOR
    \STATE \textbf{return} $\Pi$    
  \end{algorithmic}
\end{algorithm}

\begin{remark}\label{rem:nf}
  It should be noted that the form of the guards $\Gamma_{i}$ appearing in
  the output is not uniquely defined.  We produce a disjunctive normal
  form, as it is easier to interpret.  However, many equivalent expressions
  can be obtained by performing some simplification steps and in particular
  by trying to factorise the polynomials appearing in the clauses.  In the
  fairly simple examples considered in the next section, we always obtained
  an ``optimal'' form where no clause can be simplified any more.  In
  larger examples, this will not necessarily be the case and it is
  non-trivial to define what ``optimal'' actually should mean.
\end{remark}

\begin{remark}\label{rem:para}
  Many differential equations arising in applications depend on
  \emph{parameters}~$\boldsymbol{\chi}$, i.\,e.\ the polynomials $p_{a}$
  and $q_{b}$ defining the equations and inequalities of the corresponding
  basic semialgebraic system depend not only on the jet variables
  $(t,\mathbf{u},\dots,\mathbf{u}^{(\ell)})$, but in addition on some real
  parameters~$\boldsymbol{\chi}$.  Such situations can still be handled by
  Algorithm~\ref{alg:realsing}.  A straightforward solution consists of
  considering the parameters as additional unknown functions $\mathbf{u}$
  and adding to the given semialgebraic system the differential equations
  $\dot{\boldsymbol{\chi}}=0$ (of course, one can this way also
  incorporated easily conditions on the parameters like positivity
  constraints by adding corresponding inequalities).

  However, it is easier to apply directly Algorithm~\ref{alg:realsing} with
  only some trivial modifications.  We consider $p_{a}$ and $q_{b}$ as
  elements of the polynomial ring $\mathcal{D}_{\ell}[\boldsymbol{\chi}]$.
  For the parametric Gaussian elimination, there is no difference between
  the parameters $\boldsymbol{\chi}$ and the jet variables
  $(t,\mathbf{u},\dots,\mathbf{u}^{(\ell)})$: all of them represent
  parameters of the linear system of equations \eqref{eq:vess} for the
  Vessiot spaces.  Thus in the output of the elimination step, the
  guards~$\gamma_{\tau}$ will now generaly depend on both the jet variables
  $(t,\mathbf{u},\dots,\mathbf{u}^{(\ell)})$ and the additional parameters
  $\boldsymbol{\chi}$, i.\,e.\ they will also be defined in terms of
  polynomials in $\mathcal{D}_{\ell}[\boldsymbol{\chi}]$.  Hence the guards
  $\gamma_{\tau}$ returned in the second line of
  Algorithm~\ref{alg:realsing} will be defined by such polynomials, too. In
  the fifth line, we still consider only the existential closure over the
  jet variables $(t,\mathbf{u},\dots,\mathbf{u}^{(\ell)})$.  The outcome of
  the satisfiability check is now either ``unsatisfiable'' or a formula
  over the remaining parameters $\boldsymbol{\chi}$.  The only change in
  the algorithm is that in the latter case we must augment $\Gamma_{\tau}$
  by the obtained formula (and recompute a disjunctive normal form).  Note
  that the guards produced by the parametric Gaussian elimination always
  consist only of equations and inequations.  By contrast, the possibly
  appearing additional satisfiability conditions on the
  parameters~$\boldsymbol{\chi}$ are produced by a quantifier elimination
  and can be arbitrary inequalities.
\end{remark}

\section{Computational Experiments}\label{sec:computations}

We will now study the practical applicability and quality of results of the
approach developed in this article on several examples.  To this end, we
have realised a prototype implementation of Algorithm~\ref{alg:pgauss} and
Algorithm~\ref{alg:realsing} in \textsc{Reduce} \cite{Hearn:67a,Hearn:05a},
which is not yet ready for publication.  We chose \textsc{Reduce} because
on the one hand it is an open-source general purpose computer algebra
system, and on the other hand its \textsc{Redlog} package
\cite{DolzmannSturm:97a,Sturm:06a,Sturm:07a} provides a suitable
infrastructure for computations in interpreted first-order logic as
required by our approach.  Although technically a ``package'',
\textsc{Redlog} establishes a quite comprehensive software system on top of
\textsc{Reduce}. Systematically developed and maintained since 1995, it has
received more than 400 citations in the scientific literature, mostly for
applications in the sciences and in engineering. Its current code base
comprises around 65~KLOC.

Our implementation of Algorithm~\ref{alg:pgauss} uses from \textsc{Redlog}
fast and powerful simplification techniques for quantifier-free formulas
over the reals for the realisation of a nontrivial deduction procedure
$\vdash_K$.  We specifically apply the standard simplifier for ordered
fields originally described in \cite[Sect.~5.2]{DolzmannSturm:97c}; one
notable improvement since is the integration of identification and special
treatment of positive variables as a generalisation of the concept of
positive quantifier elimination described in
\cite{SturmWeber:08a,SturmWeber:09a}.  Our implementation of
Algorithm~\ref{alg:realsing} uses -- corresponding to line 5 --
implementations of real quantifier elimination, specifically virtual
substitution \cite{Kosta:16a} and partial cylindrical algebraic
decomposition \cite{Seidl:06a} as a fallback option when exceeding degree
limits for virtual substitution.

The presented examples were chosen for their simplicity allowing for any
easy check of the results with hand calculations and for the possibility to
apply also the complex algorithm of \cite{lrss:gsade} for comparison
purposes. They do not represent real benchmarks testing the feasibility of
the presented approach for large scale problems. However, they already
demonstrate the potential of our approach to concisely and explicitly
provide interesting insights into the appearance of singularities of
ordinary differential equations.  On a standard laptop, the required
computing times were on the scale of milliseconds.  We will report timings
for more serious problems elsewhere.

\begin{example}
  We continue with Example \ref{ex:sphere}, the unit sphere as first-order
  differential equation $\mathcal{J}_{1}$, and show the results of an
  automatised analysis.  Our implementation returns for the corresponding
  semialgebraic differential system
  $\Sigma_{1}=(\dot{u}^{2}+u^{2}+t^{2}-1=0)$ as input a list with three
  pairs:
  \begin{align*}
   (\Gamma_1,\Eta_1) &= \left( \dot{u} \neq 0 \, \wedge \, \dot{u}^2 + u^2 +
                    t^2 -1=0 , \,  
   \{ a =r_1, \, b=  - r_1 ( u + t \dot{u} )^{-1}  \} \right), \\
    (\Gamma_2,\Eta_2) &= \left(  t \neq 0 \, \wedge \, u^2+t^2-1 =0 \, \wedge
                     \, \dot{u}=0 , \,  
   \{ a=0 , \, b = r_2  \} \right), \\
   (\Gamma_3,\Eta_3) &= \left( t=0 \, \wedge \, u^2-1 =0 \, \wedge \,
                    \dot{u}=0 , \,  
   \{ a=r_3 , \, b = r_4 \} \right).
  \end{align*}
  It is easily seen that each guard $\Gamma_i$ describes a semialgebraic
  subset $\mathcal{J}_{1,i}\subset\mathcal{J}_{1}$ and that these sets are
  pairwise disjoint.  Each set $\Eta_i$ parametrises the Vessiot spaces at
  the points of $\mathcal{J}_{1,i}$ and one can easily read off their
  dimensions.  At each point on $\mathcal{J}_{1,1}$, the dimension is
  clearly one, since $\Eta_{1}$ contains one free variable $r_{1}=a$.  The
  dimension of the Vessiot space at each point of $\mathcal{J}_{1,2}$ is
  also one because of the free variable $r_{2}=b$, but as $\Eta_{2}$
  comprises the equation $a=0$, the Vessiot spaces are everywhere vertical.
  The set $\Eta_3$ contains two free variables $r_{3}=a$, $r_{4}=b$ so that
  everywhere on $\mathcal{J}_{1,3}$ the dimension is two.  According to
  Definition~\ref{def:sing}, the points on $\mathcal{J}_{1,1}$ are regular,
  the points on $\mathcal{J}_{1,2}$ regular singular and the two points on
  $\mathcal{J}_{1,3}$ irregular singular.  Thus we exactly reproduce the
  result of the analysis by hand presented in Example \ref{ex:sphere}.
  
  Applying the complex analysis of \cite{lrss:gsade} (more precisely, a
  \textsc{Maple} implementation of it provided by one of the authors of
  \cite{lrss:gsade}) to this example, we find that the algebraic step
  yields five cases.  One of them contains no real points at all.
  Furthermore, for the regular singular points an unnecessary case
  distinction is made by treating the two points $(\pm1,0,0)$ as a special
  case.  This distinction is not due to the behaviour of the linear system
  \eqref{eq:vess}, but stems from an algebraic Thomas decomposition of the
  sphere.  If we consider only the $\mathbb{R}$-rational points in each
  case and combine the two cases describing regular singular points, the
  result coincides with the one obtained here.
\end{example}

Thus, even in such a simple example consisting only of a scalar first-order
equation, all the potential problems of applying the complex analysis of
\cite{lrss:gsade} to real differential equations already occur.  We obtain
too many cases.  Some are completely irrelevant for a real analysis, as
they do not contain real points (in some situations, it might be non
trivial to decide whether a case contains at least some real points).
Other cases are at least irrelevant for detecting singularities.
Sometimes, the underlying case distinctions are important for a further
analysis of the singularities, but often they are simply due to the Thomas
decomposition and have no intrinsic meaning.

\begin{example}
  Dara \cite{ld:singgen} resp.\ Davydov \cite{aad:normform} classified the
  possible singularities of generic scalar first order equations
  $F(t,u,\dot{u})=0$ providing normal forms for all arising cases.  One
  distinguishes two classes: \emph{folded} and \emph{gathered}
  singularities, respectively.  In this example, we consider the gathered
  class.  It is characterised by the normal form
  \begin{equation}\label{eq:gather}
    \dot{u}^{3} + \chi u\dot{u} - t = 0
  \end{equation}
  with a real parameter $\chi$.  Values $\chi>0$ correspond to the
  \emph{hyperbolic gather}, whereas values $\chi<0$ lead to the
  \emph{elliptic gather} (classically, one considers $\chi=\pm1$).  Again,
  it is straightforward to analyse \eqref{eq:gather} by hand.  The linear
  equation for the Vessiot distribution is given by
  \begin{displaymath}
    (-1+\chi\dot{u}^{2})a+(3\dot{u}^{2}+\chi u)b=0\,.
  \end{displaymath}
  Thus the singularities lie on the parabola $3\dot{u}^{2}+\chi u=0$.  In
  the hyperbolic case, we find two real irregular singularities at
  $(\mp2/\sqrt{\chi^{3}},-3/\chi^{2},\pm1/\sqrt{\chi})$ where both
  coefficients of the linear equation vanish; in the elliptic case no real
  irregular singularities exist (see Figure~\ref{fig:ellhyp}).

  \begin{figure}
    \begin{minipage}[t]{0.45\textwidth}
      \includegraphics[width=\textwidth]{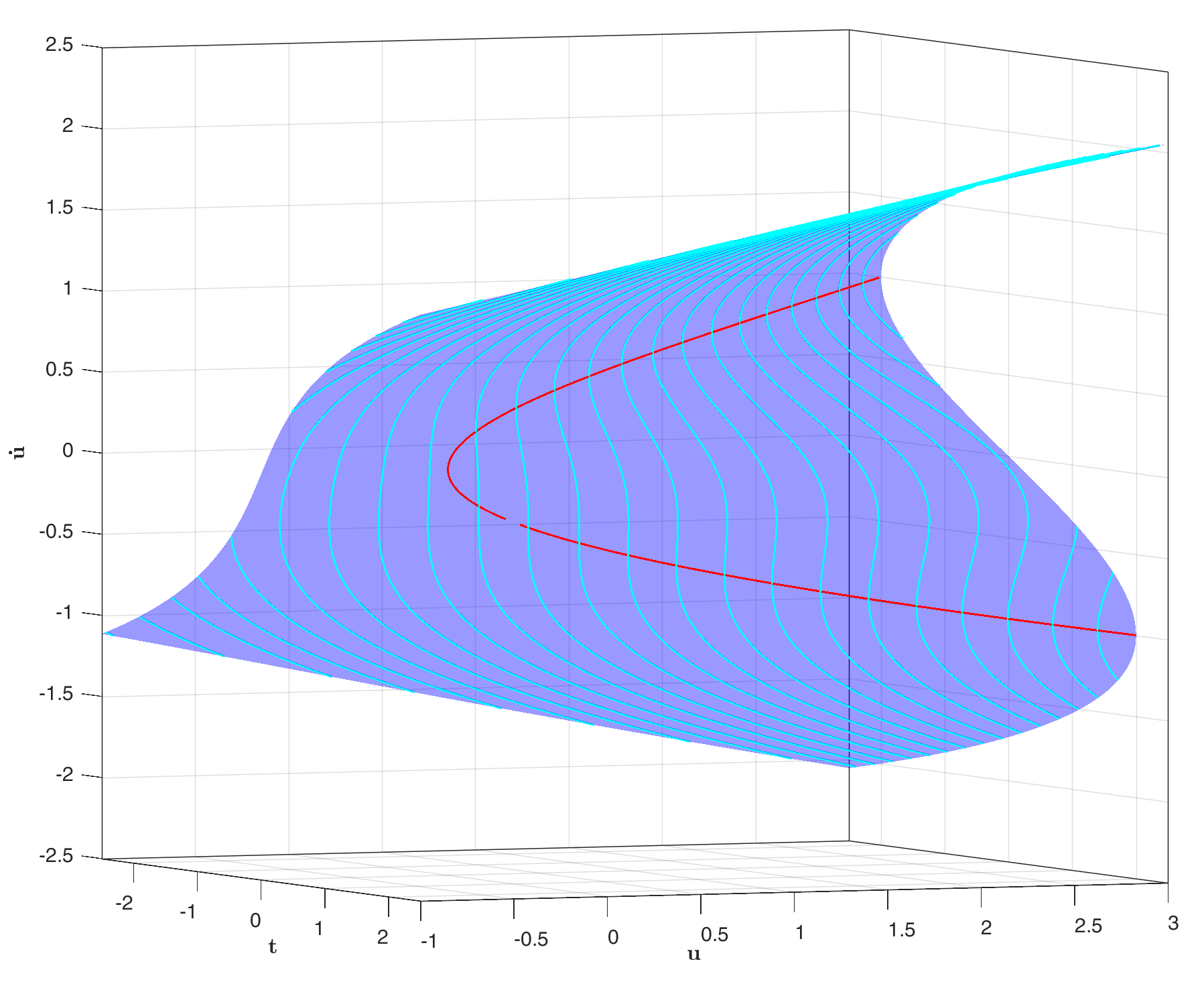}
    \end{minipage}
    \begin{minipage}[t]{0.45\textwidth}
      \includegraphics[width=\textwidth]{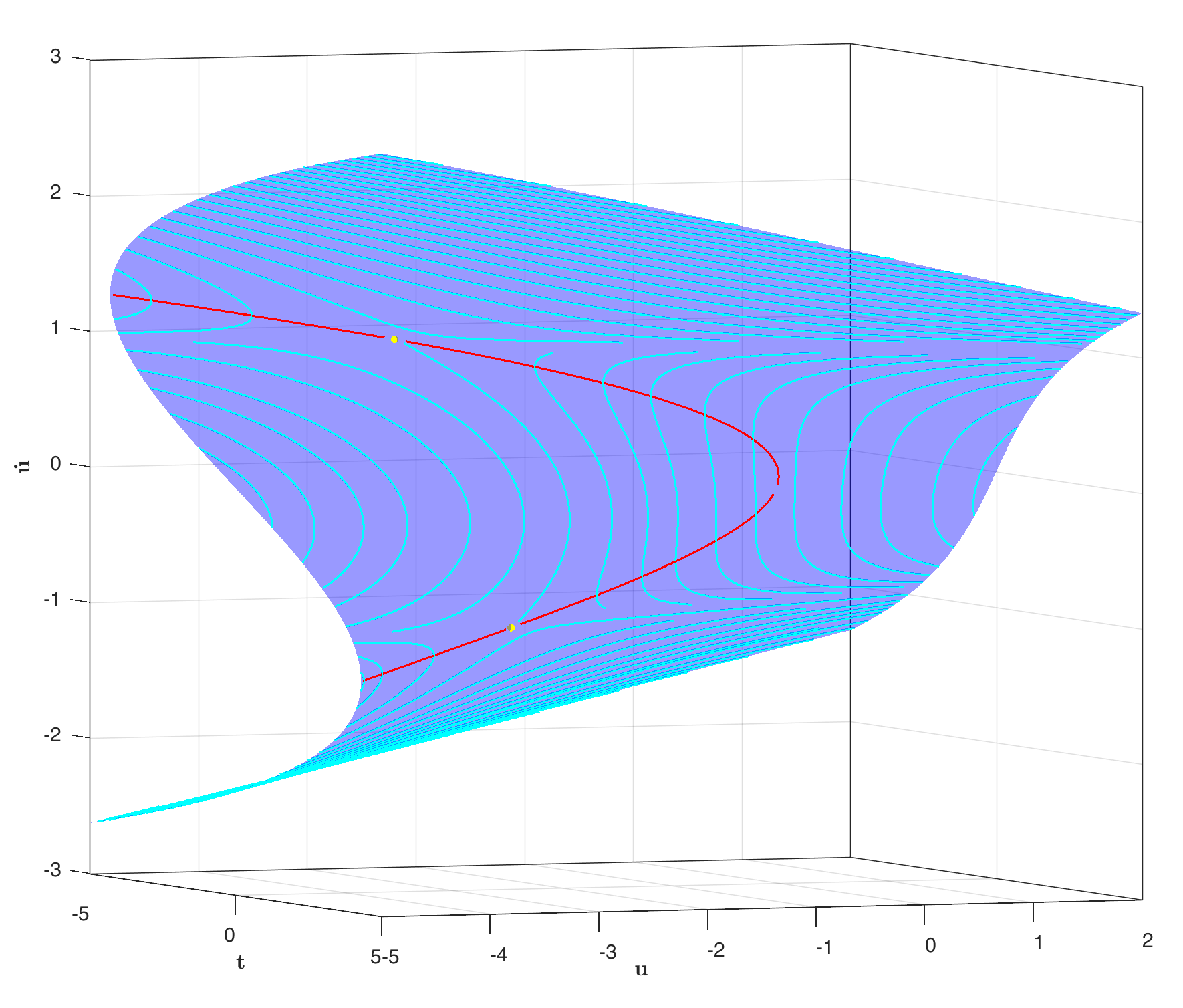}
    \end{minipage}
    \caption{Elliptic and hyperbolic gather}
    \label{fig:ellhyp}
  \end{figure}
  
  Our implementation applied to the parametric differential equation
  \eqref{eq:gather} returns three pairs:
  \begin{align*}
    (\Gamma_1,\Eta_1) &= \left( 3 \dot{u}^2 + \chi u \neq 0 \, \wedge \,
                       \dot{u}^3 + \chi \dot{u} u - t=0, \, 
    \{ b= r_1 (3 \dot{u}^2 + \chi u)^{-1} (1- \chi \dot{u}^2), \ a = r_1 \} \right), \\
    (\Gamma_2,\Eta_2) &= \left( \chi \dot{u}^{2}-1 \neq 0 \, \wedge \,  
    3 \dot{u}^2 + \chi u =0 \, \wedge \, \dot{u}^3 + \chi \dot{u} u -t=0, \, 
    \{ a=0, \ b = r_2       \} \right), \\
    (\Gamma_3,\Eta_3) &= \left( 3 \dot{u}^2 + \chi u = 0 \, \wedge \,
                       \dot{u}^3 + \chi \dot{u} u - t=0 \, 
    \wedge \, \chi \dot{u}^{2}-1=0 \, \wedge \, \chi>0 ), \, 
    \{ a= r_3, \ b = r_4 \} \right).
  \end{align*}
  As in the previous example, one can easily read off from the solutions
  $\Eta_{i}$ that the first case describes the regular points, the second case
  the regular singularities and the last case the irregular singularities.
  Note in the guard of the third case the clause $\chi>0$.  It represents
  the solvability condition for the clause $\chi \dot{u}^{2}-1=0$ and
  distinguishes between the elliptic and the hyperbolic gather.  In the
  elliptic gather the third case does not appear.

  The results of a complex analysis are independent of the value of the
  parameter $\chi$.  The algebraic Thomas decomposition yields seven cases:
  three with regular points, three with regular singularities and one with
  irregular singularities.  One of the cases with regular singularities
  never contains a real point independent of $\chi$; the existence of real
  irregular singularities depends of course on the sign of $\chi$.  The
  other unnecessary case distinctions stem again from an algebraic Thomas
  decomposition of the given equation.
\end{example}

So far, we have always studied each differential equation in the jet bundle
of the order of the equation.  However, in some cases it is also of
interest to study prolongations, i.\,e.\ to proceed to higher order.  This
is e.\,g.\ necessary to see whether solutions of finite regularity exist
(for a detailed analysis of a concrete class of quasilinear second-order
equations in this respect see \cite{wms:quasilin}).  Obviously, the
regularity of solutions is an issue only over the real numbers, as any
holomorphic function is automatically analytic.  A natural question is then
whether there exists a maximal prolongation order at which all
singularities can be detected.  The following example due to
Lange-Hegermann \cite[Ex.~2.93]{lh:phd} shows that this is not the case, as
in it at any prolongation order something new happens.  We make here
contact with some classical (un)decidability questions for power series
solutions of differential equations as e.\,g.\ studied in the classical
article by Denef and Lipshitz \cite{dl:pss}.

\begin{example}
  We start with the first-order equation $\mathcal{J}_{1}\subset J_{1}\pi$
  in three unknown functions $u$, $v$, $w$ of the independent variable $t$
  defined by the following polynomial system:
  \begin{equation}\label{eq:lh1}
    tv\dot{u}-tu+1=0\,,\quad \dot{v}-w=0\,,\quad \dot{w}=0\,.
  \end{equation}
  To obtain the first prolongation $\mathcal{J}_{2}\subset J_{2}\pi$, we
  must augment the system \eqref{eq:lh1} by the equations
  \begin{displaymath}
    tv\ddot{u}+(tw+v-t)\dot{u}-u=0\,,\quad \ddot{v}=\ddot{w}=0\,.
  \end{displaymath}
  If we prolong further to some order $q>2$, then for the definition of
  $\mathcal{J}_{q}\subset J_{q}\pi$ we must add for each integer
  $2< k\leq q$ the three equations
  \begin{displaymath}
    tvu^{(k)} + \bigl[(k-1)(tw+v)-t\bigr]u^{(k-1)} +
    (k-1)\bigl[(k-2)w-1\bigr]u^{(k-2)}=0\,,
    \quad v^{(k)}= w^{(k)}=0\,.
  \end{displaymath}

  The Vessiot spaces of $\mathcal{J}_{1}$ arise as solutions of the linear
  system
  \begin{displaymath}
    (tw+v-t)\dot{u}a + tvb_{u}=0\,,\quad b_{v}=b_{w}=0\,.
  \end{displaymath}
  For computing the Vessiot spaces of the prolonged equation, we exploit
  Proposition~\ref{prop:prolvess} telling us that at each prolongation
  order only the newly added equations must be considered.  Hence we always
  obtain a linear system containing three equations.  At any prolongation
  order $q>1$, the Vessiot spaces of $\mathcal{J}_{q}$ are defined by the
  linear system
  \begin{displaymath}
    \Bigl[\bigl(q(tw+v)-t\bigr)u^{(q)} + q\bigl((q-1)w-1\bigr)u^{(q-1)}\Bigr]a +
    tvb_{u}=0\,,\quad b_{v}=b_{w}=0\,.
  \end{displaymath}

  We fed the basic semialgebraic systems $\Sigma_{1}$, $\Sigma_{2}$ and
  $\Sigma_{3}$ corresponding to the first three equations
  $\mathcal{J}_{1}$, $\mathcal{J}_{2}$ and $\mathcal{J}_{3}$ into our
  implementation.  For each system, it returned three cases containing the
  regular, regular singular and irregular singular points, respectively, of
  the corresponding differential equation.  We obtained for $q=1$ the
  following results (we only discuss the guards $\Gamma_{i}$ and do not
  present the respective solutions $\Eta_{i}$).  As already mentioned in
  Remark~\ref{rem:sing}, the regular points represent the generic case and
  the corresponding guard is given by
  $\Gamma_{1}=(\Sigma_{1} \wedge v\neq0 \wedge t\neq0)$.  There is one
  family of regular singular points described by the guard
  \begin{displaymath}
    \Gamma_{2} = \bigl(\, \dot{w}=0 \wedge \dot{v}-w=0 \wedge tu-1=0
    \wedge v=0 \wedge t(w-1)\dot{u}-u\neq0\, \bigr)\,.
  \end{displaymath}
  Obviously $v=0$ is the condition characterising singularities.  The final
  inequation distinguishes the regular from the irregular ones: the guard
  $\Gamma_{3}$ for the latter one differs from $\Gamma_{2}$ only by this
  inequation becoming an equation.  For later use, we make the following
  observation.  The equation $tu-1=0$ implies that neither $t$ nor $u$ may
  vanish at a singularity.  Thus at an irregular singularity we cannot have
  $w=1$ or $\dot{u}=0$, as otherwise the final equation in $\Gamma_{3}$
  would be violated.
  
  We refrain from explicitly writing down all the guards of the next
  prolongations, as they become more and more lengthy with increasing
  order.  The regular points are always described by a guard of the form
  $\Gamma_{1}=(\Sigma_{q}\wedge v\neq0 \wedge t\neq0)$.  The key condition
  for singularities is always $v=0$.  Besides the equations from
  $\Sigma_{2}$, the guard $\Gamma_{2}$ for the regular singularities of
  $\mathcal{J}_{2}$ contains in addition the equation $t(w-1)\dot{u}-u=0$
  and the inequation $t(2w-1)\ddot{u}+2(w-1)\dot{u}\neq0$ whereas for the
  irregular singularities this inequation becomes again an equation.  Thus
  all the singularities of $\mathcal{J}_{2}$ lie over the \emph{irregular}
  singular points of $\mathcal{J}_{1}$.  This is not surprising, as it is
  easy to see that firstly for any differential equation $\mathcal{J}_{q}$
  all singularities of its prolongation $\mathcal{J}_{q+1}$ must lie over
  the singularities of $\mathcal{J}_{q}$ and secondly that the fibre over a
  regular singular point is always empty.  This time we can observe that at
  an irregular singularity we cannot have $w=1/2$ or $\ddot{u}=0$.  The
  results of $\mathcal{J}_{3}$ are in complete analogy: now $w=1/3$ or
  $u^{(3)}=0$ are not possible at an irregular singularity.

  The above made observations are of importance for the (non-)existence of
  formal power series solutions.  Assume that we want to construct such a
  solution for the initial conditions $u(t_{0})=u_{0}$, $v(t_{0})=v_{0}$
  and $w(t_{0})=w_{0}$.  Recall that a point in the jet bundle $J_{q}\pi$
  corresponds to a Taylor polynomial of degree $q$.  Thus a point $\rho$ on
  a differential equation $\mathcal{J}_{q}$ may be considered as such a
  Taylor polynomial approximating a solution.  This Taylor polynomial can
  be extended to one of degree $q+1$, if and only if the prolonged equation
  $\mathcal{J}_{q+1}$ contains at least one point lying over $\rho$.  As
  already mentioned, this is never the case, if $\rho$ is a regular
  singularity.  Hence, there can never exist a formal power series solution
  through a regular singular point.  Our observations have now the
  following significance.  Assume that we choose $v_{0}=0$ so that we are
  always at a singularity.  Then no formal power series solution exists, if
  we choose $w_{0}=1$, as the $w$-coordinate of an irregular singularity of
  $\mathcal{J}_{1}$ can never have the value $1$.  Similarly, no formal
  power series solutions exists for $w_{0}=1/2$, but now the problem occurs
  at the prolonged equation $\mathcal{J}_{2}$ where the $w$-coordinate of
  an irregular singularity can never have the value $1/2$.  Generally, one
  can show by a simple induction that for $w_{0}=1/k$ with $k\in\NN$ no
  formal power series solution exists, as the prolongation
  $\mathcal{J}_{k}$ of order $k$ does not contain a corresponding irregular
  singularity.
\end{example}

\begin{example}
  As a final example, we study a minor variation of \eqref{eq:lh1} which
  destroys most of the interesting properties of \eqref{eq:lh1}, but which
  nicely demonstrates why it is useful to take some care with how the
  guards are returned.  We consider the following basic semialgebraic
  system which differs from \eqref{eq:lh1} only by a missing factor $t$ in
  one term:
  \begin{equation}\label{eq:lh2}
    tv\dot{u}-u+1=0\,,\quad \dot{v}-w=0\,,\quad \dot{w}=0\,.
  \end{equation}
  While our implementation yields for the regular points exactly the same
  guard as before, the dropped factor leads to considerable more distinct
  cases of regular and irregular singularities.  The irregular
  singularities of $\mathcal{J}_{1}$ form the union of four two-dimensional
  (real) algebraic varieties, as one can easily recognise from the
  corresponding guard in disjunctive normal form:
  \begin{align*}
    \Gamma_{3} =
    &(\, \dot{w}=0 \wedge w-1=0 \wedge \dot{v}-1=0 \wedge v=0
         \wedge u-1=0\, ) \vee {}\\
    &(\, \dot{w}=0 \wedge \dot{v}-w=0 \wedge v=0 \wedge \dot{u}=0
         \wedge u-1=0\, ) \vee {}\\
    &(\, \dot{w}=0 \wedge \dot{v}-w=0 \wedge v=0 \wedge u-1=0 \wedge t=0\, )
         \vee {} \\
    &(\, \dot{w}=0 \wedge \dot{v}-w=0 \wedge \dot{u}=0 \wedge u-1=0 \wedge
         t=0\, )\,. 
  \end{align*}
  The regular singularities form the union of two three-dimensional
  varieties without the above described union of four two-dimensional
  varieties.  This set is characterised by the following guard in
  disjunctive normal form:
  \begin{align*}
    \Gamma_{2} =
    &(\, \dot{w}=0 \wedge \dot{v}-w=0 \wedge v=0 \wedge u-1=0 \wedge
         w-1\neq0 \wedge \dot{u}\neq0 \wedge t\neq0\, ) \vee {}\\
    &(\, \dot{w}=0 \wedge \dot{v}-w=0 \wedge u-1=0 \wedge t=0 \wedge
         v\neq0 \wedge \dot{u}\neq0\, )\,. 
  \end{align*}
  As in the last example, we also considered the first two prolongations of
  $\mathcal{J}_{1}$.  The dimensions of the semialgebraic sets containing
  the regular, regular singular and irregular singular points are in any
  prolongation order $4$, $3$ and $2$.  However, the guards $\Gamma_{2}$
  and $\Gamma_{3}$ are getting more and more complicated.  For
  $\mathcal{J}_{2}$ the guard $\Gamma_{2}$ contains four conjunctive
  clauses and $\Gamma_{3}$ six; for $\mathcal{J}_{3}$ these numbers raise
  to six and eight.  Without some simplifications and the consequent
  transformation into disjunctive normal form, the guards would be much
  harder to read.  The disjunctive normal form allows for a simple
  interpretation as union of basic semialgebraic sets (not necessarily
  disjoint).
\end{example}

\section{Conclusions}

For the basic existence and uniqueness theory of explicit ordinary
differential equations, it makes no difference whether one works over the
real or over the complex numbers.  The standard proofs of the
Picard--Lindel\"of Theorem are independent of the base field.  The
situation changes completely, if one performs a deeper analysis of the
equations and if one studies more general equations admitting
singularities.  Both the questions asked and the techniques used differ
considerably over the real and over the complex numbers.  We mentioned
already in Section~\ref{sec:computations} the question of the regularity of
solutions appearing only in a real analysis.  There is a long tradition in
studying the singularities of linear ordinary differential equations (see
\cite{ww:aeode} for a rather comprehensive account of the classical results
or \cite{hz:monogr} for an advanced modern presentation) and a satisfactory
theory requires methods from complex analysis like monodromy groups and
Stokes matrices.  By contrast, singularities of nonlinear ordinary
differential equations are mostly studied over the real numbers using
methods from dynamical systems theory and differential topology (see
\cite{via:geoode,aor:bising} for an introduction and
\cite{ld:singgen,aad:normform} for some typical classification results).

In this article, we were concerned with the algorithmic detection of all
geometric singularities of a given system of algebraic ordinary
differential equations.  Using the geometric theory of differential
equations, we could reduce this problem to a purely algebraic one.  In
\cite{lrss:gsade}, two of the authors presented together with collaborators
a solution over the complex numbers via the Thomas decomposition.  Now, we
complemented the results of \cite{lrss:gsade} by developing an alternative
approach to the algebraic part of \cite{lrss:gsade} (as the part where the
base field really matters) applicable over the real numbers using
parametric Gaussian elimination and quantifier elimination.

A key novelty of this alternative approach is to consider the decisive
linear system \eqref{eq:vess} determining the Vessiot spaces first
independently of the given differential system.  This allows us to make
maximal use of the linearity of \eqref{eq:vess} and to apply a wide range
of heuristic optimisations.  Compared with the more comprehensive approach
of \cite{lrss:gsade}, this also leads to an increased flexibility and we
believe that the new approach will be in general more efficient in the
sense that fewer cases will be returned.  Although we cannot prove this
rigorously, already the comparatively small examples studied in
Section~\ref{sec:computations} show this effect.  We expect it to be much
more pronounced for larger systems, as in the approach of \cite{lrss:gsade}
it cannot be avoided that the Thomas decomposition also analyses the
geometry of a differential equation $\mathcal{J}_{\ell}$ even where it is
irrelevant for the detection of singularities.

Our main tool for this first step is parametric Gaussian elimination.  We
proposed here a variant with two specific properties required by our
application to differential equations.  Firstly, it provides a disjoint
partitioning of the parameter space.  Secondly, it takes the relative
position of the solution space with respect to a prescribed cartesian
subspace taken into account.  The last property was realised by an adapted
pivoting strategy.  Our elimination algorithm \texttt{ParametricGauss}
makes strong use of a deduction procedure $\vdash_{K}$ for efficient
heuristic tests for the vanishing or non-vanishing of certain coefficients
under the current assumptions, thus avoiding redundant case distinctions at
an early stage at comparatively little computational costs.  The practical
performance of the algorithm depends decisively on the power of this
procedure.  In our proof-of-concept realisation, we used with the
\textsc{Redlog} simplifier a well-established powerful deduction procedure.

In the examples studied here, the results always turned out to be optimal
in the sense that the output contained exactly three different cases
corresponding to regular, regular singular and irregular singular points.
In general, this will not be the case.  In more complicated examples it may
for instance happen that at different regular points different pivots are
chosen by the parametric Gaussian elimination so that these points appear
in different cases.  Sometimes there may exist an intrinsic geometric
reason for this, but sometimes these case distinctions may be simply due to
the heuristics used to choose the pivots.

In the second step of our approach, the test whether the various cases
found by the algorithm \texttt{ParametricGauss} actually appear on the
analysed differential equation $\mathcal{J}_{\ell}$ requires a quantifier
elimination.  As in practice many algebraic differential equations are as
polynomials of fairly low degree, fast virtual substitution techniques will
often suffice.  As fallback a partial cylindrical algebraic decomposition
can be used.

We have ignored algebraic singularities, i.\,e.\ singular points in the
sense of algebraic geometry.  The Jacobian criterion allows us to identify
them easily using linear algebra.  In \cite{lrss:gsade}, the detection of
algebraic and geometric singularities is done in one go.  This approach
leads again to certain redundancies, as among the algebraic singularities
case distinctions are made because of the behaviour of the linear system
\eqref{eq:vess}, although the latter is not overly meaningful at such
points.  Our novel approach is more flexible and in it we believe that it
makes more sense to separate the detection of the algebraic singularities
from the detection of the geometric singularities.

One should note a crucial difference between the real and the complex case
concerning algebraic singularities. On a complex variety, a point is
nonsingular, if \emph{and only if} a local neighbourhood of it looks like a
complex manifold \cite[Thm.~7.4]{kk:eag}.  For this reason, nonsingular
points are often called smooth.  Over the reals, one has no longer an
equivalence: there may exist singular points on a real variety around which
the variety looks like a real manifold \cite[Rem.~7.8]{kk:eag}
\cite[Ex.~3.3.12]{bcr:realag}.  At such points, both the Zariski tangent
space and the smooth tangent space are defined with the former being of
higher dimension.  For defining the Vessiot space at such a point, it
appears preferable to use the smooth tangent space.  However, it is a
non-trivial task to identify such points. Diesse \cite{md:locrealag}
presented recently a criterion for detecting them, but its effectivity is
yet unclear.  We will discuss elsewhere in more detail how one can cope
with algebraic singularities over the reals.

\bibliography{RealSing.bib}

\begin{thebibliography}{10}

\bibitem{via:geoode}
V.I. Arnold.
\newblock {\em Geometrical Methods in the Theory of Ordinary Differential
  Equations}.
\newblock Springer, 2nd edition, 1988.

\bibitem{agv:sing1}
V.I. Arnold, S.M. Gusejn-Zade, and A.N. Varchenko.
\newblock {\em Singularities of Differentiable Maps {I}: The Classification of
  Critical Points, Caustics and Wave Fronts}.
\newblock Monographs in Mathematics 82. Birkh\"auser, Boston, 1985.

\bibitem{BaaderNipkow:98a}
F.~Baader and T.~Nipkow.
\newblock {\em Term Rewriting and All That}.
\newblock Cambridge University Press, 1998.

\bibitem{bglr:thomas}
T.~B{\"a}chler, V.P. Gerdt, M.~Lange-Hegermann, and D.~Robertz.
\newblock Algorithmic {T}homas decomposition of algebraic and differential
  systems.
\newblock {\em J. Symb. Comput.}, 47:1233--1266, 2012.

\bibitem{BallarinKauers:04a}
C.~Ballarin and M.~Kauers.
\newblock Solving parametric linear systems: An experiment with constraint
  algebraic programming.
\newblock {\em ACM SIGSAM Bulletin}, 38:33--46, 2004.

\bibitem{Bareiss:68a}
E.H. Bareiss.
\newblock Sylvester's identity and multistep integer-preserving {G}aussian
  elimination.
\newblock {\em Math. Comp.}, 22(103):565--578, 1968.

\bibitem{bn:realrad}
E.~Becker and R.~Neuhaus.
\newblock Computation of real radicals of polynomial ideals.
\newblock In F.~Eyssette and A.~Galligo, editors, {\em Computational Algebraic
  Geometry}, Progress in Mathematics~109, pages 1--20. Birkh\"auser, Basel,
  1993.

\bibitem{bcr:realag}
J.~Bochnak, M.~Conte, and M.F. Roy.
\newblock {\em Real Algebraic Geometry}.
\newblock Ergebnisse der Mathematik und ihrer Grenzgebiete~36. Springer-Verlag,
  Berlin, 1998.

\bibitem{Brown:2007:CQE:1277548.1277557}
C.W. Brown and J.H. Davenport.
\newblock The complexity of quantifier elimination and cylindrical algebraic
  decomposition.
\newblock In {\em Proc. ISSAC 2007}, pages 54--60. ACM, 2007.

\bibitem{Collins:75}
G.E. Collins.
\newblock Quantifier elimination for the elementary theory of real closed
  fields by cylindrical algebraic decomposition.
\newblock In {\em Automata Theory and Formal Languages. 2nd GI Conference},
  volume~33 of {\em LNCS}, pages 134--183. Springer, 1975.

\bibitem{CollinsHong:91}
G.E. Collins and H.~Hong.
\newblock Partial cylindrical algebraic decomposition for quantifier
  elimination.
\newblock {\em J. Symb. Comput.}, 12:299--328, 1991.

\bibitem{clo:iva}
D.~Cox, J.~Little, and D.~O'Shea.
\newblock {\em Ideals, Varieties, and Algorithms}.
\newblock Undergraduate Texts in Mathematics. Springer-Verlag, New York, 4th
  edition, 2015.

\bibitem{ld:singgen}
L.~Dara.
\newblock Singularit\'es g\'en\'eriques des \'equations diff\'erentielles
  multiformes.
\newblock {\em Bol. Soc. Bras. Mat.}, 6:95--128, 1975.

\bibitem{aad:normform}
A.A. Davydov.
\newblock Normal form of a differential equation, not solvable for the
  derivative, in a neighborhood of a singular point.
\newblock {\em Func. Anal. Appl.}, 19:81--89, 1985.

\bibitem{dl:pss}
J.~Denef and L.~Lipshitz.
\newblock Power series solutions of algebraic differential equations.
\newblock {\em Math. Ann.}, 267:213--238, 1984.

\bibitem{md:locrealag}
M.~Diesse.
\newblock On local real algebraic geometry and applications to kinematics.
\newblock Preprint arXiv:1907.12134, 2019.

\bibitem{DolzmannSturm:97b}
A.~Dolzmann and T.~Sturm.
\newblock Guarded expressions in practice.
\newblock In W.~K{\"u}chlin, editor, {\em Proc. ISSAC 1997}, pages 376--383.
  ACM, 1997.

\bibitem{DolzmannSturm:97a}
A.~Dolzmann and T.~Sturm.
\newblock {REDLOG}: Computer algebra meets computer logic.
\newblock {\em ACM SIGSAM Bulletin}, 31:2--9, 1997.

\bibitem{DolzmannSturm:97c}
A.~Dolzmann and T.~Sturm.
\newblock Simplification of quantifier-free formulae over ordered fields.
\newblock {\em J. Symb. Comput.}, 24:209--231, 1997.

\bibitem{vpg:decomp}
V.P. Gerdt.
\newblock On decomposition of algebraic {PDE} systems into simple subsystems.
\newblock {\em Acta Appl. Math.}, 101:39--51, 2008.

\bibitem{glhr:tdds}
V.P. Gerdt, M.~Lange-Hegermann, and D.~Robertz.
\newblock The {MAPLE} package {TDDS} for computing {T}homas decompositions of
  systems of nonlinear {PDE}s.
\newblock {\em Comp. Phys. Comm.}, 234:202--215, 2019.

\bibitem{gg:stable}
M.~Golubitsky and V.W. Guillemin.
\newblock {\em Stable Mappings and Their Singularities}.
\newblock Graduate Texts in Mathematics 14. Springer-Verlag, New York, 1973.

\bibitem{Grigoriev:88a}
D.Yu. Grigoriev.
\newblock Complexity of deciding {T}arski algebra.
\newblock {\em J. Symb. Comput.}, 5:65--108, 1988.

\bibitem{Hearn:67a}
A.C. Hearn.
\newblock \textsc{Reduce}---a user-oriented system for algebraic
  simplification.
\newblock {\em ACM SIGSAM Bulletin}, 1:50--51, 1967.

\bibitem{Hearn:05a}
A.C. Hearn.
\newblock {REDUCE}: The first forty years.
\newblock In A.~Dolzmann, A.~Seidl, and T.~Sturm, editors, {\em Algorithmic
  Algebra and Logic: Proceedings of the A3L 2005}, pages 19--24. BOD,
  Norderstedt, Germany, 2005.

\bibitem{eh:degen}
E.~Hubert.
\newblock Detecting degenerate behaviors in first order algebraic differential
  equations.
\newblock {\em Theor. Comp. Sci.}, 187:7--25, 1997.

\bibitem{ja:lec}
M.~Janet.
\newblock {\em {L}e\c{c}ons sur les {S}yst\`emes d'\'{E}quations aux
  {D}\'eriv\'ees {P}artielles}.
\newblock Cahiers Scientifiques, Fascicule {IV}. Gauthier-Villars, Paris, 1929.

\bibitem{wms:aims}
U.~Kant and W.M. Seiler.
\newblock Singularities in the geometric theory of differential equations.
\newblock In W.~Feng, Z.~Feng, M.~Grasselli, X.~Lu, S.~Siegmund, and J.~Voigt,
  editors, {\em Dynamical Systems, Differential Equations and Applications
  (Proc. 8th AIMS Conference, Dresden 2010)}, volume~2, pages 784--793. AIMS,
  2012.

\bibitem{kk:eag}
K.~Kendig.
\newblock {\em Elementary Algebraic Geometry}.
\newblock Graduate Texts in Mathematics~44. Springer-Verlag, New York, 1977.

\bibitem{ko:daag}
E.R. Kolchin.
\newblock {\em Differential Algebra and Algebraic Groups}.
\newblock Academic Press, New York, 1973.

\bibitem{Kosta:16a}
M.~Ko{\v s}ta.
\newblock {\em New Concepts for Real Quantifier Elimination by Virtual
  Substitution}.
\newblock Doctoral dissertation, Saarland University, Germany, 2016.

\bibitem{lh:phd}
M.~Lange-Hegermann.
\newblock {\em Counting Solutions of Differential Equations}.
\newblock PhD thesis, RWTH Aachen, Germany, 2014.
\newblock Available at
  \url{http://darwin.bth.rwth-aachen.de/opus3/frontdoor.php?source_opus=4993}.

\bibitem{lrss:gsade}
M.~Lange-Hegermann, D.~Robertz, W.M. Seiler, and M.~Sei{\ss}.
\newblock Singularities of algebraic differential equations.
\newblock Preprint Kassel University (arXiv:2002.11597), 2020.

\bibitem{rn:realrad2}
R.~Neuhaus.
\newblock Computation of real radicals of polynomial ideals {II}.
\newblock {\em J. Pure Appl. Alg.}, 124:261--280, 1998.

\bibitem{olv:lgde}
P.J. Olver.
\newblock {\em Applications of {L}ie Groups to Differential Equations}.
\newblock Graduate Texts in Mathematics 107. Springer-Verlag, New York, 1986.

\bibitem{aor:bising}
A.O. Remizov.
\newblock A brief introduction to singularity theory.
\newblock Lecture Notes, SISSA, Trieste, 2010.

\bibitem{riq:edp}
C.~Riquier.
\newblock {\em Les Syst\`emes d'\'Equations aux Deriv\'ees Partielles}.
\newblock Gauthier-Villars, Paris, 1910.

\bibitem{ritt:da}
J.F. Ritt.
\newblock {\em Differential Algebra}.
\newblock Dover, New York, 1966.
\newblock (Original: AMS Colloquium Publications, Vol. XXXIII, 1950).

\bibitem{dr:habil}
D.~Robertz.
\newblock {\em Formal Algorithmic Elimination for {PDEs}}.
\newblock Lecture Notes in Mathematics 2121. Springer, Cham, 2014.

\bibitem{Seidl:06a}
A.~Seidl.
\newblock {\em Cylindrical Decomposition Under Application-Oriented Paradigms}.
\newblock Doctoral dissertation, Universit{\"a}t Passau, Germany, 2006.

\bibitem{wms:invol}
W.M. Seiler.
\newblock {\em Involution: {T}he Formal Theory of Differential Equations and
  its Applications in Computer Algebra}.
\newblock Algorithms and Computation in Mathematics~24. Springer, Berlin, 2010.

\bibitem{wms:quasilin}
W.M. Seiler and M.~Sei{\ss}.
\newblock Singular initial value problems for scalar quasi-linear ordinary
  differential equations.
\newblock Preprint Kassel University (arXiv:2002.06572), 2018.

\bibitem{Sit:92a}
W.Y. Sit.
\newblock An algorithm for solving parametric linear systems.
\newblock {\em J. Symb. Comput.}, 13:353--394, 1992.

\bibitem{ss:realrad}
S.~Spang.
\newblock On the computation of the real radical.
\newblock Diploma thesis, Technical University Kaiserslautern, Department of
  Mathematics, 2007.

\bibitem{Sturm:06a}
T.~Sturm.
\newblock New domains for applied quantifier elimination.
\newblock In {\em Proc. CASC 2006}, volume 4194 of {\em LNCS}. Springer, 2006.

\bibitem{Sturm:07a}
T.~Sturm.
\newblock \textsc{Redlog} online resources for applied quantifier elimination.
\newblock {\em Acta Academiae Aboensis, Ser.~B}, 67(2):177--191, 2007.

\bibitem{Sturm:17a}
T.~Sturm.
\newblock A survey of some methods for real quantifier elimination, decision,
  and satisfiability and their applications.
\newblock {\em Math. Comput. Sci.}, 11(3--4):483--502, 2017.

\bibitem{Sturm:18a}
T.~Sturm.
\newblock Thirty years of virtual substitution.
\newblock In {\em Proc. ISSAC 2018}, pages 11--16. ACM, 2018.

\bibitem{SturmWeber:08a}
T.~Sturm and A.~Weber.
\newblock Investigating generic methods to solve {H}opf bifurcation problems in
  algebraic biology.
\newblock In {\em Proc. Algebraic Biology 2008}, volume 5147 of {\em LNCS},
  pages 200--215. Springer, 2008.

\bibitem{SturmWeber:09a}
T.~Sturm, A.~Weber, E.O. Abdel-Rahman, and M.~{El Kahoui}.
\newblock Investigating algebraic and logical algorithms to solve {H}opf
  bifurcation problems in algebraic biology.
\newblock {\em Math. Comput. Sci.}, 2(3):493--515, 2009.

\bibitem{th:ds}
J.M. Thomas.
\newblock {\em Differential Systems}.
\newblock Colloquium Publications XXI. AMS, 1937.

\bibitem{th:sr}
J.M. Thomas.
\newblock {\em Systems and Roots}.
\newblock W. Byrd Press, 1962.

\bibitem{ww:aeode}
W.~Wasow.
\newblock {\em Asymptotic Expansions for Ordinary Differential Equations}.
\newblock Dover, New York, 1965.

\bibitem{Weispfenning:88a}
V.~Weispfenning.
\newblock The complexity of linear problems in fields.
\newblock {\em J. Symb. Comput.}, 5:3--27, 1988.

\bibitem{Weispfenning:97b}
V.~Weispfenning.
\newblock Quantifier elimination for real algebra---the quadratic case and
  beyond.
\newblock {\em Appl. Algebr. Eng. Comm.}, 8:85--101, 1997.

\bibitem{hz:monogr}
H.~{\.{Z}}o{\l}{\k{a}}dek.
\newblock {\em The Monodromy Group}.
\newblock Monografie Matematyczne~67. Birkh\"auser, Basel, 2006.

\end{thebibliography}
\bibliographystyle{plain}

\subsection*{Acknowledgments}

This work was supported by the bilateral project ANR-17-CE40-0036 and
DFG-391322026 SYMBIONT.  The first two authors thank Marc Diesse for useful
discussions about singularities in real algebraic geometry.  The third
author is grateful to Sarah Sturm at the University of Bonn and Marco Voigt
at Max Planck Institute for Informatics for inspiring and clarifying
discussions around Propositions \ref{prop:termination} and \ref{prop:cd}.

\end{document}